\documentclass[12pt,reqno]{amsart}
\usepackage{amsmath,amssymb,amsfonts,amscd,latexsym,amsthm,mathrsfs,verbatim,textcomp,bm,cite}
\usepackage[colorlinks,linkcolor=black,citecolor=black]{hyperref}
\usepackage{cite}
\usepackage{hypbmsec}
\usepackage{enumerate}
\usepackage{bm}
\usepackage{tikz}
\usetikzlibrary{matrix,arrows,decorations.pathmorphing}

\def\le{\leqslant}
\def\ge{\geqslant}
\def\leq{\leqslant}
\def\geq{\geqslant}

\def\Case#1#2{%
\smallskip\paragraph{\textbf{\boldmath Case #1: #2}}\hfil\break\ignorespaces}

\newtheorem{theorem}{Theorem}
\newtheorem{corollary}[theorem]{Corollary}

\newtheorem{lemma}[theorem]{Lemma}
\newtheorem{prop}[theorem]{Proposition}
\newtheorem{remark}[theorem]{\bf Remark}

\renewcommand{\Im}{\operatorname{Im}}

\newcommand{\ord}{\operatorname{ord}}
\newcommand{\Vol}{\operatorname{Vol}}
\newcommand\cL{\mathcal{L}}
\newcommand\cF{\mathcal{F}}
\newcommand\cW{\mathcal{W}}

\newcommand{\Z}{\mathbb{Z}}
\newcommand{\Q}{\mathbb{Q}}
\newcommand{\R}{\mathbb{R}}

\newcommand{\N}{\mathbb{N}}

\newcommand{\e}{\operatorname{e}}

\newcommand{\sumthree}{\operatorname*{\sum\sum\sum}}

\usepackage{etoolbox}

\numberwithin{equation}{section}

\numberwithin{theorem}{section}

\newcommand{\QQ}{\mathbb{Q}}
\newcommand{\Fq}{\mathbb{F}_q}

\def\cA{{\mathcal A}}

\def\cF{{\mathcal F}}
\def\cG{{\mathcal G}}

\def\cI{{\mathcal I}}
\def\cJ{{\mathcal J}}

\def\cL{{\mathcal L}}
\def\cM{{\mathcal M}}

\def\cP{{\mathcal P}}
\def\cQ{{\mathcal Q}}
\def\cR{{\mathcal R}}
\def\cS{{\mathcal S}}

\def\cU{{\mathcal U}}
\def\cV{{\mathcal V}}
\def\cW{{\mathcal W}}

\def\cZ{{\mathcal Z}}

\def\newsim{\approx} 
\def\sfB{\mathsf {B}}

\def \balpha{\bm{\alpha}}
\def \bbeta{\bm{\beta}}

\def\ov\QQ{\overline{\QQ}}

\def\newY{L}

\def\e{\mathbf{e}}
\def\eq{{\mathbf{\,e}}_q}
\def\em{{\mathbf{\,e}}_m}

\def\mand{\qquad\mbox{and}\qquad}

\usepackage{mathtools}
\usepackage{todonotes}
\usepackage[norefs,nocites]{refcheck}

\def\({\left(}
\def\){\right)}
\def\fl#1{\left\lfloor#1\right\rfloor}
\def\rf#1{\left\lceil#1\right\rceil}

\begin{document}

\title[Modular square roots of primes] 
{Bilinear forms in Weyl sums for modular square roots and applications}

\dedicatory{\it Dedicated to Bruce Berndt and his penchant for Gauss sums, \\ on the occasion of his 80th birthday.}

 \author[A. Dunn]{Alexander Dunn} 
 \address{A.D.: Department of Mathematics, University of Illinois at
Urbana-Cham\-paign  1409 West Green Street, Urbana, IL 61801, USA}
 \email{ajdunn2@illinois.edu}
 \author[B.\ Kerr]{Bryce Kerr}
\address{B.K.: Department of Mathematics and Statistics,
University of Tur\-ku,  FI-20014, Finland}
\email{bryce.kerr@utu.fi}
 \author[I.~E.~Shparlinski]{Igor E. Shparlinski}
 \address{I.E.S.: School of Mathematics and Statistics, University of New South Wales.
 Sydney, NSW 2052, Australia}
 \email{igor.shparlinski@unsw.edu.au}
 
 \author[A. Zaharescu]{Alexandru Zaharescu} 
 \address{A.Z.: Department of Mathematics, University of Illinois at
Urbana-Cham\-paign  1409 West Green Street, Urbana, IL 61801, USA and 
Simon Stoilow Institute of Mathematics of the Romanian Academy, P.O. Box 1-764, RO-014700 Bucharest, Romania}
 \email{zaharesc@illinois.edu}
 
 \begin{abstract}  
Let $q$ be a prime, $P \geq 1$ and let $N_q(P)$ denote the number of rational primes $p \leq P$ that split in the imaginary quadratic field
$\mathbb{Q}(\sqrt{-q})$. The first part of this paper establishes various unconditional and conditional (under existence of a Siegel zero) lower bounds for
$N_q(P)$ in the range $q^{1/4+\varepsilon} \leq P \leq q$, for any fixed $\varepsilon>0$. This improves upon what is implied by work of Pollack and Benli--Pollack.

The second part of this paper is dedicated to proving an estimate for a bilinear form involving Weyl sums for modular square roots (equivalently
Sali\'{e} sums). Our estimate has a power saving in the  so-called P{\'o}lya--Vinogradov range, and our methods involve studying an additive energy coming from 
quadratic residues in $\mathbb{F}_q$. 

This bilinear form is inspired by the recent \textit{automorphic} motivation: the second moment for twisted $L$-functions
attached to Kohnen newforms has  recently been computed by the first and fourth authors. So the third part 
of this paper links the above two directions together and outlines 
the \textit{arithmetic} applications of this bilinear form. These include the equidistribution of quadratic roots of primes,
products of primes, and relaxations of a conjecture of Erd{\H o}s--Odlyzko--S{\'a}rk{\"o}zy.
\end{abstract}

\keywords{prime quadratic residues, modular square roots, discrepancy, bilinear sums, additive energy, sieve, binary forms}
\subjclass[2010]{11A15, 11E25, 11K38, 11L07, 11L20, 11N32}
  
\maketitle
\tableofcontents
\section{Introduction}

\subsection{Motivation and description of our results} \label{sec:motiv}
Our motivation begins in the early 20th century. This is when I.~M.~Vinogradov initiated the study of the distribution of both quadratic residues and non-residues modulo a prime $q$.
This remains a central theme in classical analytic number theory. 

Several fundamental conjectures are still unresolved. Vinogradov's least quadratic
 non-residue conjecture asserts that
\begin{equation} \label{nonresi}
n_q \le q^{o(1)},
\end{equation}
where $n_q$ denotes the least non-quadratic residue modulo $q$ ($n_q$ is necessarily prime). The best unconditional bounds
known are
$$
n_q \le q^{1/(4 \sqrt{e})+o(1)},
$$
largely due to Burgess' bounds for character sums~\cite{Bur} and ideas of Vinogradov.

It is also natural to consider the least prime quadratic residue $r_q$ modulo $q$. The best 
unconditional result is due to Linnik and  Vinogradov~\cite{LV}, where they showed 
\begin{equation} \label{resp}
r_q \le q^{1/4+o(1)}.
\end{equation}
Much of the above discussion can be generalised to prime residues and non-residues 
of an arbitrary Dirichlet character $\chi$ of order $k$. One can see the work of Norton~\cite{Nor} for the analogue 
of~\eqref{nonresi} and of Elliot~\cite{Ell} for the analogue of~\eqref{resp}.

Given the existence of a small prime quadratic residue and non-residue, the next natural question to ask is how many such quadratic residues and non-residues  exist in a given interval $[2,P]$.  Making use of reciprocity relations and the sieve, Benli and Pollack~\cite{BePo} have results in this direction 
for quadratic, cubic and biquadratic residues. 
For general Dirichlet characters
one can see work of Pollack~\cite{Po2} and also a more recent work of Benli~\cite{Ben}.
The impact and links of results of this type stretch far beyond analytic number theory. 
For example, Bourgain and Lindenstrauss~\cite[Theorem~5.1]{BoLi} are motivated by links with the {\it Arithmetic Quantum Unique Ergodicity Conjecture\/}. 
They have shown that for any $\varepsilon > 0$, there exists some $\delta>0$, 
such that for any sufficiently large $D$,
the set 
$$
\cR: = \left\{p~\text{prime}:~ D^\delta \le p \le D^{1/4+\varepsilon}, \ \(\frac{D}{p}\) = -1\right\}
$$
satisfies
$$
\sum_{p\in \cR} \frac{1}{p} \ge \frac{1}{2} - \varepsilon.
$$
 When one relaxes the condition that a non-residue be prime, one does much better, 
as in the case of Banks, Garaev, Heath-Brown and Shparlinski~\cite{BGHBS}, who
 show that for each $\varepsilon>0$ and $N \geq q^{1/4\sqrt{e}+\varepsilon}$, 
 the proportion of quadratic non-residues modulo $q$ in the interval $[1,N]$ is bounded away from $0$ when $q$ is large enough. 

 Furthermore, if $q \equiv 3 \pmod{4}$ is a large prime, then quadratic reciprocity tells us that the following two 
 Legendre symbols
are equal:  $(-q/p)=(p/q)$. Thus counting 
quadratic residues modulo $q$ is equivalent to counting small rational primes $p \leq P$ that split in the imaginary quadratic field $F:=\mathbb{Q}(\sqrt{-q})$. In particular, throughout this paper, 
$$
\chi(\cdot):=\( \frac{-q}{\cdot}\)
$$
defined via the Jacobi symbol, 
always denotes the character attached to $F$, and  $L(s,\mathcal{\chi})$ denotes the Dirichlet $L$-function attached 
to $\chi$. Let $N_q(P)$ denote the number of rational primes $p\le P$ that split in $F$. 

A good reference point for the strength of our results is what the Generalised Riemann Hypothesis (GRH) implies. It is well known that under the GRH for $L(s,\chi)$
that we have   $N_q(P)  \ge c_1 P/\log P$ for $P \ge c_2 (\log q)^2$ for some absolute constants $c_1, c_2>0$, see, 
for example,~\cite[Section~13.1, Excercise~5(a)]{MoVa}.

 Furthermore, a result of Heath-Brown~\cite[Theorem~1]{HB3} immediately implies the following.  
For any fixed $\varepsilon >0$, for all but $o(Q/\log Q)$ 
primes $q \le [Q, 2Q]$, 
we have   $N_q(P) = (1/2+ o(1)) P/\log P$ for $P\ge q^\varepsilon$, as $Q \to \infty$.

The second theme which we develop in this paper concerns bounds of certain bilinear sums closely related 
to correlations between values of {\it  Sali\'{e} sums}
\begin{equation} \label{eq:Sal}
S(m,n;q)  = \sum_{x \in \mathbb{F}_q} \(\frac{x}{q} \) \e_q(mx+n \overline{x}), 
\end{equation} 
 see~\cite{Sal}. 
 We  {\it emphasise\/}
that this is closely related to recent 
work of the first and fourth authors~\cite{DuZa}, 
 who have computed a second moment for $L$-functions 
attached to a half-integral weight Kohnen newform,
averaged over all primitive characters modulo a prime. 
Power savings in the error term for such a moment come in part 
 from savings in the bound on correlations 
between Sali\'{e} sums~\eqref{eq:Sal}. Our argument gives a direct improvement of~\cite[Theorem~1.2]{DuZa},
see Appendix~\ref{app:C}. 
It remains to investigate whether  this improvement propagates into  a quantitative improvement
in the error term  of~\cite[Theorem~1.1]{DuZa} for the  second moment of the above $L$-functions.

Let $M,N$ be two positive real numbers and
$\balpha = (\alpha_m)_{m\sim M}$  and $\bbeta =  (\beta_n)_{n\sim N}$ be complex weights
supported on dyadic intervals $m\sim M$ and  $n\sim N$, where  $a \sim A$   indicates
$a \in [A, 2A)$. Let $K: \mathbb{F}_q \rightarrow \mathbb{C}$ be some function, usually called a {\it  kernel}. A bilinear form involving $K$  is a sum of the shape 
$$
\sum_{m\sim M} \sum_{n\sim N}  \alpha_m \alpha_n K(mn).
$$
Bounds of such sums also have key  automorphic and arithmetic applications. 

In a series of two recent breakthrough papers using deep algebro--geometric techniques, Kowalski, Michel and Sawin~\cite{KMS1,KMS2}  have established non-trivial estimates 
for bilinear forms  with Kloosterman sums
 in~\cite{KMS1} and  generalised Kloosterman sums in~\cite{KMS2}.  
In particular, their estimates apply below  the {\it P{\'o}lya--Vinogradov
range\/}, that is, when the ranges of summation are $M, N \sim q^{1/2}$;
 this is where completion and Fourier theoretic methods breakdown. Such bounds are a crucial ingredient
to the evaluation of asymptotic moments of $\operatorname{GL}_2(\mathbb{A}_{\mathbb{Q}})$ $L$-functions over primitive Dirichlet characters, with power saving error (as well as many more automorphic applications), see~\cite{BFKMM,KMS1,KMS2,Zac} and references therein.

It is important to note that the bilinear sums we study here
 do not fall under the umbrella of the results of~\cite{KMS2}. 
The initial approach of~\cite{KMS2} (Vinogradov's $ab$ shifting trick and the Riemann Hypothesis for algebraic curves over a finite field) can be used.
However,
our approach leads to a much stronger result, for comparison see Appendix~\ref{app:B}.  For an alternative treatment of bilinear forms
in classical Kloosterman sums using the \textit{sum-product} phenomenon, see~\cite{Shk}.

The above two themes:
\begin{itemize}
\item {\it lower bounds on the number of prime quadratic residues and non-residues;\/}
\item {\it bounds of bilinear sums with modular square roots;\/}
\end{itemize}
come together in the third direction which we pursue here: {\it  the distribution of  
square roots modulo $q$ of primes $p\le P$\/}. Indeed, in the asymptotic formula we obtain, the main term
 is controlled by 
the counting function of prime quadratic residues,
while the error term is given by the discrepancy, and depends on the quality of our bounds for
certain bilinear sums we estimate in this paper.

We remark that it is natural to attempt to improve the error term on the average square of some $L$-functions 
from~\cite{DuZa}. This however requires substantial effort with optimisation and balancing a 
number of estimates and so falls outside of the scope of this paper. 

We are now ready to state some of the results in this paper. A high level sketch of the ideas and methodology in the proofs is deferred to Section~\ref{highlevel}.

\subsection{Counting split primes in imaginary quadratic extensions}  \label{sec:split}
Our first result is an 
unconditional lower bound for $N_q(P)$, but with ineffective constant.
\begin{theorem} \label{uncondthm1}
For any fixed  $\varepsilon>0$, any
sufficiently large prime $q$, and any $P$ with 
$q\ge P \ge q^{1/4+\varepsilon}$,  we have 
$$
N_q(P) \geq c(\varepsilon)  \min \left\{ P^{1/2} q^{-\varepsilon/2}, Pq^{-1/4-2\varepsilon/3}\right\},
$$ 
where $c(\varepsilon)>0$ depends only on $\varepsilon$. 
\end{theorem}

\begin{remark} As with  the results on quadratic residues of Benli and Pollack~\cite{Ben,BePo,Po2}, Siegel's theorem is used in the proof of 
Theorem~\ref{uncondthm1}, and so the constant $c(\varepsilon)$ is ineffective. 
\end{remark} 

Observe that for $\varepsilon>0$ small and $A>0$ fixed~\cite[Theorem~1.3]{Po2} guarantees that 
$$
N_q(q^{1/4+\varepsilon}) \ge c(\varepsilon, A)  (\log q)^{A}.
$$
for some constant $c(\varepsilon,A)>0$ that depends only on $\varepsilon$ and  $A$. 
Theorem~\ref{uncondthm1} above improves this to a small power of $q$, that is, 
$$
N_q\(q^{1/4+\varepsilon}\) \geq c(\varepsilon)  q^{\varepsilon/3}.
$$
This was also proved independently by Benli~\cite{Ben} (taken with $k=2$) very recently. 
Theorem~\ref{uncondthm1} 
also substantially improves the lower bound 
$$
N_q(P) \geq P^{1/25} q^{-1/50}, \quad q^{1/2+\varepsilon} \leq  P \leq q,
$$
established by Benli and Pollack~\cite[Theorem~3]{BePo} and
in particular, implies
$$
N\(q^{1/2+\varepsilon} \) \ge q^{1/4}.
$$
for any $\varepsilon > 0$, provided that $q$ is large enough. 

Our next result is an unconditional bound for $N_q:=N_q(q)$, with effective constant.
\begin{theorem} \label{uncondthm2}
Suppose $q \ge 67$ is prime with  $q \equiv 3 \pmod{16}$. Then  
$$
N_q >\frac{\(2-\log (3 \sqrt{2}) \)}{2} \frac{\lfloor \sqrt{3q/4} \rfloor}{\log q}.
$$
\end{theorem} 

Since 
$$
\frac{\(2-\log (3 \sqrt{2}) \)  \sqrt{3}}{4}  = 0.2402 \ldots, 
$$
we see from Theorem~\ref{uncondthm2} that there  is an effectively computable 
absolute constant $c_0$ such that for $q\ge c_0$ we have 
$$
N_q >\ \frac{0.24  \sqrt{q} }{\log q}. 
$$
In fact one can get a better constant   by estimating certain quantities more carefully. 

\begin{remark} Unfortunately the argument of the proof of 
Theorem~\ref{uncondthm2} does not scale to estimate $N_q(P)$ with $P < q$. However 
it actually increases its strength for $P > q$, which is still a meaningful range in the problem of 
estimating  $N_q(P)$.  We do not consider this case as small values of $P$ are of our principal 
interest. 
\end{remark} 

Many famous unsolved conjectures are known to hold under the assumption of a Siegel zero~\cite{FI}. This is because one can sometimes break the parity problem of the sieve with this 
hypothesis.
 A notable example is Heath-Brown's proof~\cite{HB2} of the twin prime conjecture assuming Siegel zeros. 
 Continuing this tradition, we prove an essentially sharp lower bound for $N_q(P)$ under the assumption that  $L(s,\chi)$ has a mild Siegel zero. 

\begin{theorem} \label{siethm}
Suppose $q \equiv 3 \pmod{4}$ is a large prime and
\begin{equation}\label{Lsize}
L\(1,\chi\) = O\(1/(\log q)^{10}\).
\end{equation}
Then for any fixed  $\varepsilon>0$ and $P$ with $q^{1/2+\varepsilon} \le P \leq q$, we have 
$$
N_q(P) \ge c(\varepsilon) h(-q) \frac{P}{ \sqrt{q} (\log q)^2},
$$ 
where $h(-q)$ is the class number of $F=\Q(\sqrt{-q})$
and $c(\varepsilon)>0$ depends only on $\varepsilon$.
\end{theorem}

\begin{remark} The constant $c(\varepsilon)$ in Theorem~\ref{siethm} is ineffective. 
\end{remark} 

Counting small split primes in general number fields is a notoriously difficult and fundamental problem. Ellenberg and Venkatesh~\cite{EV} have 
established a direct connection between counts for small split primes and bounds for $\ell$-torsion in general class groups. 

We also refer to~\cite{Go,GrOno,GOT} for results and references on lower bounds on class numbers of imaginary quadratic fields. 

\subsection{Bilinear forms and equidistribution}
We recall that $a \sim A$ means $a \in [A, 2A)$. Given two real numbers $M,N$ and complex weights 
\begin{equation}
\label{eq:weight}
\balpha = (\alpha_m)_{m\sim M}  \quad \text{and} \quad \bbeta =  (\beta_n)_{n\sim N},
\end{equation}
we denote 
$$
\| \balpha \|_{\infty}:=\max_{m \sim M} |\alpha_m| \quad \text{and} \quad \| \balpha \|_{\sigma}:= \(\sum_{m \sim M} |\alpha_m|^{\sigma} \)^{\frac{1}{\sigma}},
$$
and similarly for $\bbeta$. 

For $a,h \in \Fq^{\times}$ we consider bilinear forms in Weyl sums for square roots
\begin{equation} \label{Wdef}
W_{a,q}(\balpha, \bbeta; h,M,N):=   \sum_{m \sim M}  \sum_{n \sim N} \alpha_m \beta_n \sum_{\substack{x \in \Fq \\
x^2 = amn}} \eq(hx).
\end{equation}
We notice that this is equivalent to studying bilinear forms with Sali\'{e} sums, given by~\eqref{eq:Sal}, thanks to the
 evaluation  in~\cite{Sal}, see also~\cite[Lemma~12.4]{IwKow} and~\cite[Lemma~4.4]{SA}:
\begin{equation}
\label{eq:S eval}
\frac{1}{\sqrt{q}}  S(m,n;q)   =  \frac{1}{\sqrt{q}}  S(1,mn;q)   =\varepsilon_q  \( \frac{n}{q} \)  \,  \sum_{\substack{x \in \mathbb{F}_q \\ x^2=mn}} \eq \( 2x \) , 
\end{equation}
where $ \varepsilon_q  = 1$ if $q \equiv 1 \pmod 4$ and  $ \varepsilon_q  = i$  if $q \equiv 3 \pmod 4$ 
(note that if $(mn/q) = -1$, the Sali\'{e} sum vanishes).
The tuple of characters $\(1,(\cdot/q) \)$ is \textit{Kummer induced} (cf.~\cite[Section~2]{KMS2}), and so the results of~\cite{KMS2} 
do not apply to~\eqref{Wdef}.

Our goal is to improve the trivial bound 
$$ 
W_{a,q}(\balpha, \bbeta; h,M,N) = O\(  \| \balpha \|_1   \| \bbeta \|_1\),
$$
in the P{\'o}lya--Vinogradov range. In arithmetic applications the case when  weights satisfy 
\begin{equation} \label{eq:balanced}
\| \balpha  \|_{\infty}, \| \bbeta \|_{\infty} = q^{o(1)}
 \end{equation}
is most important. However,
our bounds are valid for more general weights.  
Making use of weighted additive energies, 
and some  ideas from~\cite{BGKS} and~\cite{CCGHSZ}, 
 we prove the following result.
 
\begin{theorem} \label{thm:Waq}
For any positive integers $M,N\le q/2$ and any weights $\balpha$  and $\bbeta$ as in~\eqref{eq:weight}, 
we have 
\begin{align*}
 |W_{a,q}(\balpha,\bbeta;h,M,N)|\le  \|\balpha\|_2 \|\bbeta\|_{\infty}^{1/3}& \|\bbeta\|_1^{2/3}   q^{1/8+o(1)} M^{7/24}N^{1/8}\\
 & \quad\left(\frac{M^{7/48}}{q^{1/16}}+1\right)\left(\frac{N^{7/48}}{q^{1/16}}+1\right)
\end{align*}
and
\begin{align*}
|W_{a,q}(\balpha, \bbeta; h,M,N)|   \le \| \balpha \|_2  \| \bbeta \|_1^{3/4}  & \| \bbeta \|_{\infty}^{1/4} q^{1/8+o(1)} M^{5/16} N^{1/16}\\
 &  \(\frac{M^{3/16}}{q^{1/8}}  +1\)  \(\frac{N^{3/16}}{q^{1/8}} +1 \).
\end{align*}
\end{theorem}

When the weights $\balpha$ and $\bbeta$ satisfy~\eqref{eq:balanced}, we can re-write the  bounds  of Theorem~\ref{thm:Waq} in the following simplified form 
 \begin{equation} 
 \begin{split}
 \label{eq:simple2}
|W_{a,q}(\balpha, \bbeta; h,M,N)|   &\le  q^{1/8+o(1)}   (MN)^{19/24}\\
&  \qquad  \left(\frac{M^{7/48}}{q^{1/16}}+1\right)\left(\frac{N^{7/48}}{q^{1/16}}+1\right) 
 \end{split}
\end{equation}
and
 \begin{equation} 
 \begin{split}
 \label{eq:simple2-2}
|W_{a,q}(\balpha, \bbeta; h,M,N)|   &\le  q^{1/8+o(1)}   (MN)^{13/16}\\
&  \qquad    \(\frac{M^{3/16}}{q^{1/8}}  +1\)  \(\frac{N^{3/16}}{q^{1/8}} +1 \). 
 \end{split}
\end{equation} 
A power savings in the error term of an asymptotic formula for a second moment of certain $L$-functions~\cite{DuZa} 
comes from savings in the bound on 
$W_{a,q}(\balpha, \bbeta; h,M,N)$ in the P{\'o}lya--Vinogra\-dov range, as well as other ranges in which 
spectral techniques are used.
 Additional ideas are also needed in~\cite{DuZa} to make the asymptotic formula
unconditional, because $\alpha_m$ and $\beta_n$ are coefficients of a fixed normalised
 Kohnen newform.
  It is not known yet that they satisfy~\eqref{eq:balanced} (in this context, the 
 condition~\eqref{eq:balanced} is the Ramanujan--Petersson conjecture, also equivalent to the Lindel\"{o}f hypothesis for the 
 twisted $L$-function attached to the Shimura lift).
 The best known bound is $\alpha_m = O\( m^{1/6+\varepsilon}\)$ due to Conrey and Iwaniec~\cite{CI}.

We now outline the arithmetic applications of Theorem~\ref{thm:Waq}.  

 We recall that  the {\it discrepancy\/} $D_N$ of a sequence in $\xi_1, \ldots, \xi_N \in [0,1)$    is defined as 
\begin{equation}
\label{eq:Discr}
D_N = 
\sup_{0\le \alpha < \beta \le 1} \left |  \#\{1\le n\le N:~\xi_n\in [\alpha, \beta)\} -(\beta-\alpha)  N \right |, 
\end{equation} 
where  $\# \cS$ denotes  the cardinality of  $\cS$ (if it is finite), see~\cite{DrTi,KN} for background.
For positive integers $P$ and $R$, denote the discrepancy of the sequence (multiset) of points
$$
 \{x/q:~ x^2 \equiv pr \pmod q \ \text{for some primes} \ p\le P,\ r \le R\}
$$
by $\Delta_q(P,R)$.
Combining the bound~\eqref{eq:simple2} with the classical Erd\H{o}s--Tur\'{a}n inequality, 
see   Lemma~\ref{lem:ET}  below, we derive an equidistribution of the modular square roots of products of two primes.

\begin{corollary} \label{cor:disc pr} For $1\le P,R \le q$  we have 
$$
\Delta_q(P,R)\le q^{1/8}(PR)^{19/24+o(1)}\left(\frac{P^{7/48}}{q^{1/16}}+1\right)\left(\frac{R^{7/48}}{q^{1/16}}+1\right).
$$
\end{corollary} 

Our next application is to a relaxed version of the still open problem of 
Erd{\H o}s, Odlyzko and S{\'a}rk{\"o}zy~\cite{EOS} on the representation 
of  all reduced classes modulo an integer  $m$  as the products $pr$ of
two primes $p,r \le m$. This question has turned out to be too hard even 
for the
Generalised Riemann Hypothesis, thus various relaxations 
have been considered, see~\cite{Shp} for a
short overview of currently 
available results in this direction. 

Here we 
obtain the following variant about products of two small primes and a small square. 
To simplify the result we assume that $p,r \le q^{2/3}$.  In this case, using the bound~\eqref{eq:simple2-2}  one easily derives 
similarly to  Corollary~\ref{cor:disc pr} the following result. 

\begin{corollary} \label{cor:EOS_Square}  
Let   $P,R \le q^{2/3}$ be real numbers such that 
the interval $[2,P]$ contains $P^{1+o(1)}$ of  both prime quadratic residues and non-residues. 
If $q\ge S \ge 1$ is a real number such that 
 $$
  (PR)^{3/16}    S \ge q^{9/8 + \varepsilon},
 $$
then any reduced residue class modulo $q$ 
 can be represented as $prs^2$  for two primes $p\le P$, $r \le R$ with some real $P,R \le q^{2/3}$ and a
  positive  integer $s \le  S$.
 \end{corollary}
 
 In particular, if $P= R = S$ then the result of Corollary~\ref{cor:EOS_Square} is 
 nontrivial if $P \ge q^{9/11 + \varepsilon}$, while for $P=R = q$ we need $S \ge q^{3/4 + \varepsilon}$. 

\subsection{Distribution of roots of  primes}
For a positive integer $P$  we denote the discrepancy of the sequence (multiset) of points
$$
 \{x/q:~ x^2 \equiv p \pmod q \ \text{for some prime} \ p\le P\}
$$
by $\Gamma_q(P)$.

Theorem~\ref{thm:Waq} in combination with the classical Erd\"{o}s--Tur{\'a}n inequality (see Lemma~\ref{lem:ET}) and the 
Heath-Brown identity~\cite{HB} yield the following result on the equidistribution of square roots of primes.
Since we are mostly interested in the values of $P$ which are as small as possible, we use only 
the first bound of Theorem~\ref{thm:Waq}. For large values of $P$ one can get a better result but using the second bound as well
or both, see also Remark~\ref{rem:U1U2}  below. 

\begin{theorem}
\label{thm:disc p} For any $P\le q$ we have 
$$
\Gamma_q(P) \le q^{61/1760}P^{61/66+o(1)}+q^{13/110}P^{9/11+o(1)}.  
$$ 
\end{theorem}  
 
If the interval $[2,P]$ contains  $P^{1+o(1)}$ prime quadratic residues, which is certainly expected and 
is known under the GRH,  see Sections~\ref{sec:motiv} and~\ref{sec:split}, 
then Theorem~\ref{thm:disc p}  is nontrivial  for $P \ge q^{13/20+ \varepsilon}$ with some fixed $\varepsilon > 0$. 

We also point out that the equidistribution here is the opposite situation considered 
in~\cite{DFI1,DFI2,Hom,LiuMas,To}, where the modulus is prime and varies. 

\section{High level sketch of the methods} \label{highlevel}
Here we outline the main ideas and methods behind the proofs in this paper, without paying too much attention to technical detail.
Let $q \ge P \ge q^{1/4+\varepsilon}$. The starting point in Theorem~\ref{uncondthm1} is a certain linear combination of logarithms
 that biases split primes. In particular, consider
$$
\boldsymbol{Q}(P):=\sum_{n \sim P} r(n) \log n \in \mathbb{R},
$$
where $r(n)=R_{-q}(n)/2$, and $R_{-q}(n)$ is the number of representations of $n$ by a complete 
set of inequivalent positive definite quadratic forms of discriminant $-q$.  

The behaviour of the character $\chi(n):=(-q/n)$ governs $r(n)$, the mean value of $r(n)$ and hence $\boldsymbol{Q}(P)$. 
Observe 
that by~\cite[Equation~(22.22)]{IwKow}, for $q \ge 5$ we have
\begin{equation} \label{quadrep}
R_{-q}(n)= 2  r(n)=  2 \prod_{p^{\ell} \mid \mid n}\(1+\chi(p)+\cdots+\chi^{\ell}(p) \), \quad \gcd(n,q)=1.
\end{equation}
Classical work of Linnik and Vinogradov~\cite{LV}, and subsequent refinements by Pollack~\cite{Po2} give a mean value asymptotic of the shape
\begin{equation} \label{LVasymp}
\sum_{n \leq x} r(n) \sim L(1,\chi) x, \quad x \geq q^{1/4+\varepsilon} \quad \text{and} \quad x \rightarrow \infty. 
\end{equation}
Thus 
\begin{equation} \label{Qasymp}
\boldsymbol{Q}(P) \sim \frac{1}{2} L(1,\chi) P \log P \quad \text{as} \quad P \rightarrow \infty.
\end{equation}
In reality,~\eqref{LVasymp} comes equipped with a power saving error term $O(x^{1-\eta})$ for some $\eta>0$ depending
only on $\varepsilon$, but 
our arguments require the main term to dominate. Thus an application of Siegel's theorem renders~\eqref{LVasymp} and~\eqref{Qasymp} ineffective, and so 
Theorem~\ref{uncondthm1} has an ineffective constant. It is worth pointing out that the condition $x \geq q^{1/4+\varepsilon}$ in~\eqref{LVasymp}
comes directly from Burgess bounds for character sums.

From~\eqref{quadrep}, it is clear that $r$ is a multiplicative function that is supported on powers of large split primes and is vanishing on 
powers of other primes.
Thus we would expect a large asymptotic contribution to $\boldsymbol{Q}(P)$ from integers divisible 
by powers of large split primes. 

To quantify this, one can write   
\begin{equation} \label{Qdiscuss}
\boldsymbol{Q}(P)=\sum_{p}  c_p(P) \log p, \qquad c_p(P) \in \mathbb{Z}_{\geq 0}.
\end{equation}
For a parameter $2 \leq y \leq P$, one can further consider separately the sum in~\eqref{Qdiscuss}  over split primes $p \leq y$, inert primes $2 \leq p \leq P$ and
 over split primes $y<p \leq P$. The rest of the proof chooses 
  the largest possible $y$ such that the sum over split primes $y<p \leq P$ contributes at least $\varepsilon L(1,\chi)P \log P$ to~\eqref{Qasymp}. From
there it is routine to show that this implies a lower bound of $yq^{-o(1)}$ split primes $y<p \leq P$. 

The coefficients $c_p(P)$ are approximately mean values of the $r(b)$ for $b \leq P/p$. 
Thus to bring~\eqref{Qasymp} into play, we need
$Pq^{-1/4-o(1)} \geq y$, and so we see that Burgess bounds directly impact our method. On the other hand, to 
control the asymptotic contribution of small split primes, we need $y \leq P^{1/2-o(1)}$.  See Section~\ref{uncondthm1sec} for the full
proof.

Theorem~\ref{uncondthm2} is an effective lower bound for the number of split primes $P<q$. We consider the principal form
$$
\boldsymbol{Q}(U,V)=U^2+\frac{q+1}{4} V^2. 
$$
Our strategy here is to consider the product
$$
\boldsymbol{R}_q:=\prod_{1 \leq n \leq t} \boldsymbol{Q}(n,1), \quad t=  \fl{ \frac{\sqrt{3q}}{2}}.  
$$
Each prime $p \mid \boldsymbol{R}_q$ is split and satisfies $2 \leq p \leq q$. We obtain
a lower bound on $\omega(\boldsymbol{R}_q)$ by estimating $\ord_p(\boldsymbol{R}_q)$. We use a combination of Hensel lifting and
Stirling's formula to do this. This can be found in Section~\ref{sec:uncondthm2}.

We now shift our attention to Theorem~\ref{siethm}. Now $q \ge P \ge q^{1/2+\varepsilon}$ and we want to count split primes in $[2,P]$.
A natural starting point would be to search for primes represented by one of 
the $h(-q)$ classes of binary quadratic forms. Suppose 
$$
F(U,V)=AU^2+BUV+CV^2
$$ 
is a such a reduced form, that is, 
$$
\gcd(A,B,C)=1 \quad \text{and} \quad |B| \leq A \leq C, 
$$
and $B^2-4AC=-q$. Define 
$$
\pi_F(x):=\{p \leq x: p=f(u,v) \text{ for some } (u,v) \in \mathbb{Z}^2 \}.
$$

The Chebatorev Density Theorem~\cite{TSc}, see also~\cite{LO},   implies that primes are asymptotically equidistributed amongst form classes 
\begin{equation} \label{cheb}
\pi_F(x) \sim \frac{\delta_F x }{h(-q) \log x} \quad \text{as} \quad x \rightarrow \infty,
\end{equation}
where $\delta_F=1/2$ or $1$ depending on whether $F(U,V)$ and $F(U,-V)$ are $\text{SL}_2(\mathbb{Z})$-equivalent in the class group or not.  
The asymptotic~\eqref{cheb} in~\cite{LO} requires $x$ exponentially larger than $q$, which certainly is insufficient for the range we are interested.
If one assumes the GRH for Hecke $L$-functions, then the same arguments of Lagarias and Odlyzko extend~\eqref{cheb} to the range $x 
\geq q^{1+\varepsilon}$, referred to as the \textit{GRH range}. This narrowly misses the range we are interested in.

However, the GRH range does not take into account the size of the coefficients $F$, and one might expect primes to appear with 
relative frequency when $x \geq C^{1+\varepsilon}$, referred to as the \textit{optimal range} (cf.~\cite{Zam3}). 
Zaman~\cite[Theorem~1.1]{Zam3} has established upper bounds for $\pi_F(x)$ in the optimal range under the GRH for $L(s,\chi)$,
but we are far away from any type of unconditional lower bound. One can see a recent impressive work of 
Thorner and Zaman~\cite[Theorem~1.2]{TZam}
 that in particular shows the least prime represented by $F$ is $O(q^{694})$
(improving work of~\cite{Fog,We,KM}). This was later improved 
to $O(q^{455})$ by Zaman in his PhD thesis~\cite[Corollary~1.4.1]{Zam2}.

Duke's equidistribution theorem on Heegner points on the 
modular surface~\cite{Duk} gives a positive proportion of reduced forms whose coefficients satisfy 
$\max \{ |A|,|B|,|C| \} = O\( \sqrt{q}\)$ (note that Duke's theorem contains an ineffective constant, which is one source of 
ineffectivity of Theorem~\ref{siethm}). The proof of Theorem~\ref{siethm} proceeds initially in the spirit above: 
the $\beta$-sieve, as in~\cite[Theorem~11.13]{FrIw}, is applied to each form. 
For some absolute constant $c>0$, each form produces at least $c P/\sqrt{q} \log^2 q$ 
integers with a bounded 
number of prime factors (one can also see~\cite[Theorem~1.5]{Zam3}, however we sieve in only one variable for simplicity). 

One runs immediately into the parity problem of the sieve,
and this is where the Siegel zero comes in. It is necessary for us to collect each sieved set 
produced from each form into a single set, $\mathcal{F}_{-q}$, of size 
$$
\# \mathcal{F}_{-q} \ge c h(-q) \frac{P}{\sqrt{q} \log^2 q}
$$
for some absolute constant $c>0$. 
Now $\mathcal{F}_{-q}$ is large enough for us to use the Siegel zero. For a split prime of medium size,
it remains to sieve out integers of the form $n=pm$ from $\mathcal{F}_{-q}$, where $R_{-q}(m) \geq 1$. 
Observe that~\eqref{LVasymp} tells us
 that there are fewer than normal such $n$, approximately $O\left(L(1,\chi)P/p \right)$. A similar argument
using  $\eqref{LVasymp}$ tells us there are fewer than normal split primes $p \leq P^{1/2}$, roughly $O(L(1,\chi)^{1/3} P^{1/2})$ many.
This turns out to be enough to break parity, and the full proof of Theorem~\ref{siethm} can be found Section~\ref{sec:siethm}.

The heart of Theorem~\ref{thm:Waq} relies on a non-trivial estimate for a certain weighted additive energy, which can be found in Section~\ref{sec:E-bounds}.
For a complex weight $\bbeta$ as in~\eqref{eq:weight} and $j \in \Fq^{\times}$, we study 
\begin{equation}
\label{eq:Energydef}
E_{q,j}(\bbeta):=\sum_{\substack{(u,v,x,y) \in \Fq^4 \\ u+y=x+v }} \beta_{ju^2} \overline{\beta}_{jv^2} \beta_{jx^2} \overline{\beta}_{jy^2},
\end{equation}
where $ju^2,jv^2, jx^2, jy^2$ are  all computed modulo $q$ and take the value of the reduced residue between $1$ and $q$. We omit the subscript $j$ when $j=1$.
Quantities of this type are well known in additive combinatorics under the name of {\it additive energy\/}. Algebraic manipulations reduce such a problem to counting the number
$\mathbb{F}_q$-rational points on curves of the form 
\begin{equation}
\label{eq:eqnintro}
y^2 =f(x),
\end{equation} 
with a polynomial  $ f\in \mathbb{F}_q[X]$ of degtree $\deg f=2$, 
all lying in a small box. This is a special case of a problem considered in~\cite[Theorem~5]{CCGHSZ} which uses the method of Weyl for quadratic exponential sums. Similar principles are also implicit in the recent work of the first and fourth author~\cite{DuZa} where instead of studying energies, the distribution of $\alpha n^2$ is used. Directly applying results from~\cite{CCGHSZ} gives a new energy estimate. However, we can do better using a geometry of numbers argument and some ideas from~\cite{BGKS}.  
 Assuming~\eqref{eq:eqnintro} has a large number of solutions, we construct a lattice which has large intersection with a box. Calculations with successive minima give distributional estimates on the number of $f\in \mathbb{F}_q[x]$ such that the equation~\eqref{eq:eqnintro} has a large number of solutions. This leads to a fourth moment energy estimate from which a bound for~\eqref{eq:Energydef} follows from convexity, see~\eqref{sec:E-bounds} for more details.

It is interesting to note that in the Type~I setting, that is when for the weight $\bbeta$ we have $\beta_n=1$ identically, 
Theorem~\ref{thm:Waq} (in the P{\'o}lya--Vinogradov range) surpasses
 the bound  implied by the standard argument using Vinogradov's shifting by $ab$-trick and Weil's 
 Riemann hypothesis for curves over a finite field. The details of this are worked out in Appendix~\ref{app:B},
 and are of independent interest. 

\section{Preliminaries} 

\subsection{Notation} 

Throughout the paper, the notation $U = O(V)$, 
$U \ll V$ and $ V\gg U$  are equivalent to $|U|\leqslant c V$ for some positive constant $c$, 
which throughout the paper may depend on a 
small real positive parameter $\varepsilon$. 

For any quantity $V> 1$ we write $U = V^{o(1)}$ (as $V \to \infty$) to indicate a function of $V$ which 
satisfies $|U| \le V^{\varepsilon}$ for any $\varepsilon> 0$, provided $V$ is large enough. 

For a real $A> 0$ we write $a \sim A$ to indicate  that $a$ is in the dyadic interval $A \le  a < 2A$. 

For $\xi \in \R$, and $m \in \N$ we denote
$$\e(\xi)=\exp(2 \pi i \xi) \quad \text{and} \quad \em(\xi) = \exp(2 \pi i \xi/m).
$$
We also use $(k/q)$ to denote the Legendre symbol of $k$ modulo $q$, a prime. 

We always use the letter $p$, with or without subscript, to denote a prime number. 

We use $\Fq$ to denote the finite field of $q$ elements, which we assume to be represented 
by the set $\{0, \ldots, q-1\}$ for $q$ prime.

As usual, for an integer $a$ with $\gcd(a,q)=1$ we define $\overline{a}$ by the conditions
$$
a \overline{a}  \equiv 1 \pmod q \mand \overline{a} \in \{1, \ldots, q-1\}.
$$

We also use $\mathbf{1}_{\cS}$ to denote the characteristic function of a set $\cS$, 
and  as we have mentioned, we also write $\# \cS$ for the cardinality of  $\cS$ (if it is finite). 

\subsection{Exponential sums and discrepancy}

We start with recalling  the classical Erd\H{o}s--Tur\'{a}n inequality linking the discrepancy to exponential sums (see, for instance,~\cite[Theorem~1.21]{DrTi} or~\cite[Theorem~2.5]{KN}). 

\begin{lemma}
\label{lem:ET}
Let $\xi_n$, $n\in \N$,  be a sequence in $[0,1)$. Then for any $H\in \N$, the discrepancy $D_N$ 
given by~\eqref{eq:Discr}, we have 
$$
D_N \le 3 \left( \frac{N}{H+1} +\sum_{h=1}^{H}\frac{1}{h} \left| \sum_{n=1}^{N} \e(h\xi_n) \right | \right).
$$
\end{lemma}

It  is now  useful to recall the definition of the {\it Gauss sum\/}  for odd prime modulus
\begin{equation} \label{gausseval}
\cG_q(a,b):= \sum_{x \in \Fq} \e_q \(ax^2+bx \), \qquad (a,b) \in\Fq^\times  \times \Fq.
\end{equation}
The standard evaluation in~\cite[Section~1.5]{BEW} or~\cite[Section~2]{Dav} or~\cite[Theorem~3.3]{IwKow} leads to the formula
\begin{equation} \label{eval}
\cG_q(a,b)=\e_q \(-\overline{4a} b^2 \) \varepsilon_{q} \sqrt{q} \( \frac{a}{q} \)
\end{equation} 
where 
$$
 \varepsilon_{q} =\begin{cases}
1 \quad & \text{if } q\equiv 1 \pmod 4,\\
i  \quad & \text{if } q\equiv -1 \pmod 4.
\end{cases}
$$

We also need the following bound for exponential sums over square roots.
\begin{lemma}
\label{lem:Incom Sqrt}
For any integer $W \le q$ and integers $a$ and $h$ with $\gcd(ah,q)=1$, we have 
$$
 \sum_{w = 1}^W \sum_{\substack{x \in \Fq \\ x^2= aw}} \e_q(hx) \ll q^{1/2+o(1)}.
$$
\end{lemma} 

\begin{proof} Completing the exponential sum  as in~\cite[Section~12.2]{IwKow} gives that
\begin{equation}
 \label{eq: Compl}
 \sum_{w = 1}^W\sum_{\substack{x \in \Fq \\ x^2=a w}} \e_q(hx) 
  \ll \log q \max_{0 \leq t \leq q-1} \left|\cG_q(t \overline{a},h) \right | . 
\end{equation}  
The value $t=0$ does not contribute anything  and for any $t \in\Fq^\times $ use~\eqref{eval}.   
\end{proof} 

\subsection{Siegel's theorem}

We recall that the celebrated result of Siegel gives the following lower bound on $L(1,\chi)$,  see~\cite[Chapter~21)]{Dav}. 

\begin{lemma}
\label{lem:Siegel}
For any $\delta>0$, there is a  constant $C(\delta)> 0$ such that 
$$
L(1,\chi) \ge C(\delta) q^{-\delta}.
$$
\end{lemma} 

Ineffectiveness in our results comes from Lemma~\ref{lem:Siegel},
since the constant $C(\delta)$ is ineffective.

 \section{Counting small split primes unconditionally: Proofs of Theorems~\ref{uncondthm1} and~\ref{uncondthm2}} 
 \subsection{Preparations} 
Here we recall an asymptotic formula due to Pollack~\cite[Proposition~3.1]{Po2} that builds on some work of 
Linnik and Vinogradov~\cite[Theorem~2]{LV}. It is used extensively in the proofs of 
Theorems~\ref{uncondthm1} and~\ref{siethm}. Recalling
that $\chi(\cdot):=(-q/\cdot)$, let 
$$
r(n):=\sum_{d \mid n} \chi(d).
$$

 \begin{lemma} \label{lem:Pol: polprop}   
For each $\varepsilon>0$, there is a constant $\eta >0$ for which the following holds: if $x \geq q^{1/4+\varepsilon}$, then the sum 
$$
\sum_{n \leq x} r(n)=L(1,\chi) x+O(x^{1-\eta}).
$$
where the implied constant depends only on $\varepsilon$. 
\end{lemma} 

Thus by Lemma~\ref{lem:Pol: polprop} and~\cite[Equations~(22.14) and~(22.22)]{IwKow} we have 
\begin{multline} \label{quadpol}
\sum_{\substack{n \leq x \\ \gcd(n,q)=1} } R_{-q} (n)=2 L(1,\chi) x+O(x^{1-\eta}), \\
 \quad \text{for} \quad x \geq q^{1/4+\varepsilon} \quad \text{and} \quad q \geq 5,
\end{multline}
where the implied constant depends only on $\varepsilon$. 

\subsection{Ineffective lower bounds: Proof of Theorem~\ref{uncondthm1}} \label{uncondthm1sec}

\subsubsection{Preliminaries for the proof}
We say that $n \newsim P$ if and only if  $n \in [P/2,P] \cap \mathbb{N}$
(we recall that  $n \sim P$ is defined to mean $n \in [P,2P] \cap \mathbb{N}$, 
so $n \newsim P$ is equivalent to $n \sim P/2$). 

We now consider the sum
\begin{equation} \label{eq:Q(P)}
\boldsymbol{Q}(P):=\sum_{n \newsim P}  r(n) \log n.
\end{equation}

Observe that Lemma~\ref{lem:Pol: polprop}  immediately implies that  
\begin{equation}\label{mainasym}
\boldsymbol{Q}(P)= \frac{1}{2}L(1,\chi)P \log P+O \(P^{1-\eta}  \).
\end{equation}
for some   $\eta > 0$ depending only on $\varepsilon > 0$. 
By Siegel's theorem, see Lemma~\ref{lem:Siegel}, 
in a similar way as before, we see the main term in~\eqref{mainasym} dominates. Thus  
$$
\boldsymbol{Q}(P)= \(\frac{1}{2}+o(1) \)L(1,\chi)P \log P.
$$

We now observe that each summand is positive. Factoring each  $n \newsim P$ and collecting together contributions from each prime $p \le P$ we write~\eqref{eq:Q(P)} as 
$$
\boldsymbol{Q}(P)=\sum_{2 \leq p \leq P} c_p(P) \log p,  
$$
with some integer coefficients  $c_p(P) \ge 0$. 

We fix some parameter  $y \geq 1$  to be chosen later. In order to use Lemma~\ref{lem:Pol: polprop} 
in the coming arguments, we need  
\begin{equation} \label{con1}
P/y \geq q^{1/4+2\varepsilon/3}.
\end{equation}
We write $\boldsymbol{Q}(P)$ as 
$$
\boldsymbol{Q}(P)=\boldsymbol{Q}_1(P,y)+ \boldsymbol{Q}_2(P,y)+\boldsymbol{Q}_3(P,y),
$$
where 
$$
\boldsymbol{Q}_1(P,y):=\sum_{\substack{p \leq y \\ \chi(p)=1}} c_p(P) \log p, \quad \boldsymbol{Q}_2(P,y):=\sum_{\substack{2 \leq p \leq P \\ \chi(p)=-1}} c_p(P) \log p ,
$$
and 
$$
\boldsymbol{Q}_3(P,y):=\sum_{\substack{y<p \leq P \\ \chi(p)=1}} c_p(P) \log p.
$$

Below we obtain upper bounds on the sums  $\boldsymbol{Q}_1(P,y)$ and $\boldsymbol{Q}_2(P,y)$,
which together with~\eqref{mainasym} implies a lower bound on  $\boldsymbol{Q}_3(P,y)$. 
Thus estimating $c_p(P)$, we consequently obtain a lower bound on the size of the support of 
$\boldsymbol{Q}_3(P,y) $. 

\subsubsection{Estimate for $\boldsymbol{Q}_1(P,y)$}
Let  $\ord_p n$ denote the $p$-adic order of $n$. 
Hence each $n \in \N$ can be uniquely written as 
\begin{equation} \label{eq: ell b}
n =   b p^{\ell} \mand \ell =  \ord_p n. 
\end{equation}
Then~\eqref{quadrep} implies
\begin{equation} \label{rineq}
r(n)=r(p^{\ell} b)=(\ell+1) r(b) \quad \text{if} \quad \chi(p)=1.
\end{equation}
When $\ell=1$ for a split prime, the $\ell+1=2$ on the right side of~\eqref{rineq} heavily influences the choice of $y$ (see also~\eqref{con2} below) that we eventually make. 

For a given prime $p \leq y$ with $\chi(p)=1$, we compute $c_p(P)$. Clearly for each $n \newsim P$ written in the 
form~\eqref{eq: ell b},  
using~\eqref{rineq}, we see that its   total contribution    to $c_p(P)$ 
 is
$$
\ell r(p^{\ell} b)  =  \ell(\ell+1)  r(b).
$$ 
Hence, if  $\chi(p)=1$ then we have 
$$
c_p(P) =  \sum_{\ell=1}^\infty \ell(\ell+1) \sum_{\substack{ b \newsim P/p^{\ell} \\ \gcd(b,p)=1 }} r(b). 
$$
At this point we   drop the conditions 
$$
\chi(p)=1  \mand   \gcd(b,p)=1, 
$$
which is possible by the nonnnegativity of  $r(b)$. 
We do this partially for typographical simplicity, but more 
importantly because we reuse the same  argument to estimate $\boldsymbol{Q}_2(P,y)$ below. 

Hence we write  
\begin{equation}
 \begin{split}
 \label{cp}
\boldsymbol{Q}_1(P)& \le \sum_{p \leq y}  \log p \sum_{\ell=1}^\infty \ell(\ell+1) \sum_ {b \newsim P/p^{\ell}} r(b) \\
 & = \sum_{\ell=1}^\infty \ell(\ell+1)  \sum_{p \leq y}   \log p  \sum_ {b \newsim P/p^{\ell}} r(b)  . 
\end{split}
\end{equation}

We now estimate the right side of~\eqref{cp}.  First consider the $\ell=1$ term in~\eqref{cp}. We derive 
\begin{equation} \begin{split} \label{maint}
2 \sum_{p \leq y} \log p \sum_{b \newsim P/p} r(b)  = 2 \(\frac{1}{2} L(1,\chi) P+O \(  P ^{1-\eta} y^{\eta} \)  \) \sum_{p \leq y} \frac{ \log p}{p} ,
\end{split} 
\end{equation}
for some  $\eta> 0$ depending only on $\varepsilon$, 
where the last equality follows from Lemma~\ref{lem:Pol: polprop} (which applies with $2\varepsilon/3$ instead of $\varepsilon$, 
since $y$ satisfies~\eqref{con1}). Recalling the {\it Mertens formula\/},
see~\cite[Equation~(2.14)]{IwKow}, we obtain
\begin{equation} \label{speccon}
\sum_{p \leq y } \frac{\log p}{p}= \(1 +o(1) \) \log y.
\end{equation}
We now impose a second constraint   
\begin{equation} \label{con2}
y \leq P^{1/2-\varepsilon/2}.
\end{equation} 
Note   that by~\eqref{con2}  we have 
$P ^{1-\eta} y^{\eta}  \le P^{1-(1+\varepsilon) \eta/2}$. Thus we have
$$
P ^{1-\eta} y^{\eta} \log P \ll P^{1- \eta/2} . 
$$
Thus~\eqref{maint} and~\eqref{speccon}   tell us that
\begin{equation} 
\begin{split}  
\label{l1} 
2 \sum_{p \leq y}  & \log p  \sum_{b \newsim P/p} r(b)\\
& \le  \(  \frac{1-\varepsilon}{2} +o(1) \)  L(1,\chi) P \log P  + O \(P^{1- \eta/2} \).
\end{split}
\end{equation}

Now consider the summands in~\eqref{cp} corresponding to  $\ell \geq 2$. 

First we switch the order of summations between $p$ and $\ell$ again and move the summation 
over $p$ outside. 
We then consider contributions from the terms with
$$P/p^{\ell} \geq q^{1/4+\varepsilon/3} \mand P/p^{\ell} < q^{1/4+\varepsilon/3}
$$ 
separately.

\Case{A}{$P/p^{\ell} \geq q^{1/4+\varepsilon/3}$} Reducing $\eta$ if necessary (to correspond  to  $\varepsilon/3$ rather 
than to $2\varepsilon/3$ as in our original choice), we see that  Lemma~\ref{lem:Pol: polprop}  implies that 
\begin{equation} 
\begin{split} \label{eq:useful}
 \sum_{p \leq y}\log p & \sum_{\substack{ \ell \geq 2 \\ P/p^{\ell} \geq q^{1/4+\varepsilon/3}}}  \ell(\ell+1)  \sum_{b \newsim P/p^{\ell}} r(b)  \\
&=   \sum_{p \leq y}\log p  \sum_{\substack{ \ell \geq 2 \\ P/p^{\ell} \geq q^{1/4+\varepsilon/3}}}  \ell(\ell+1)   \\
&  \qquad \qquad \qquad  \(\frac{1}{2} L(1,\chi) P/p^{\ell} +O  \((P/p^{\ell})^{1-\eta} \)  \) . 
\end{split}
\end{equation}
Furthermore
\begin{align*}
\sum_{\substack{ \ell \geq 2 \\ P/p^{\ell} \geq q^{1/4+\varepsilon/3}}}  \ell(\ell+1) & \(\frac{1}{2} L(1,\chi) P/p^{\ell} +O  \((P/p^{\ell})^{1-\eta} \)  \)\\
& \ll   L(1,\chi) P \sum_{\ell \geq 2}  \frac{ \ell(\ell+1)}{p^\ell}  +    P^{1-\eta}  \sum_{\ell \geq 2}    \frac{ \ell(\ell+1)}{p^{\ell(1-\eta)} } \\
& \le   L(1,\chi) P p^{-2+o(1)}   +  P^{1-\eta}   p^{-2(1-\eta)+o(1)}.   
\end{align*}

Now,  dropping  the primality condition (and assuming without 
loss of generality that $\eta < 1/2$),  we derive from~\eqref{eq:useful} that
\begin{equation} 
\begin{split} \label{l2}
 \sum_{p \leq y}\log p & \sum_{\substack{ \ell \geq 2 \\ P/p^{\ell} \geq q^{1/4+\varepsilon/3}}}  \ell(\ell+1)  \sum_{b \newsim P/p^{\ell}} r(b)   \\
& \le     L(1,\chi) P \sum_{k \leq y}  k^{-2+o(1)} +    P^{1-\eta}  \sum_{k \leq y} k^{-2(1-\eta)+o(1)}    \\
&\ll  L(1,\chi) P+ P^{1-\eta} .
\end{split}
\end{equation}

\Case{B}{$P/p^{\ell} < q^{1/4+\varepsilon/3}$} 
For these  terms, we first observe that from~\eqref{quadrep}  and   the well-known bound 
on the divisor function $\tau(b)$, see~\cite[Equation~(1.81)]{IwKow}, we have  
\begin{equation} 
 \label{eq: r and tau}
0 \le  r(b) \leq \tau(b) = b^{o(1)}. 
\end{equation}

We now drop the condition  $\gcd(b,p)=1$ and derive 
\begin{align*}
 \sum_{p \leq y}\log p  \sum_{\substack{ \ell \geq 2 \\ P/p^{\ell}<q^{1/4+\varepsilon/3}}} &  \ell(\ell+1) \sum_{b \newsim P/p^{\ell}} r(b) \\
& \le \sum_{p \leq y} \log p  \sum_{\substack{ \ell \geq 2 \\ P/p^{\ell}<q^{1/4+\varepsilon/3}}}  \ell(\ell+1) 
\sum_{b \newsim P/p^{\ell}} \tau(b). 
\end{align*}
We see from~\eqref{eq: r and tau} that the 
inner sum over $b$ is  
bounded by 
$$
(P/p^{\ell})^{1+o(1)}  \le q^{1/4+\varepsilon/3 + o(1)}.
$$  
We also observe that  it vanishes unless $p^{\ell} \le P$.  Hence
\begin{align*}
 \sum_{p \leq y}\log p & \sum_{\substack{ \ell \geq 2 \\ P/p^{\ell}<q^{1/4+\varepsilon/3}}}  \ell(\ell+1) \sum_{b \newsim P/p^{\ell}} r(b) \\
& \le  q^{1/4+\varepsilon/3+o(1)}  \sum_{p \leq y} \log p  \sum_{\substack{ \ell \geq 2 \\ p^{\ell}< P}}  \ell^2 \le  y  q^{1/4+\varepsilon/3+o(1)}.  
\end{align*}
Using~\eqref{con1}  we see that  $y  q^{1/4+\varepsilon/3+o(1)}\le  P q^{-\varepsilon/3 +o(1)} \ll   P^{1-\varepsilon/4}$ and 
we arrive to the estimate
\begin{equation} 
 \label{l3}
 \sum_{p \leq y}\log p   \sum_{\substack{ \ell \geq 2 \\ P/p^{\ell}<q^{1/4+\varepsilon/3}}}  \ell(\ell+1) 
 \sum_{b \newsim P/p^{\ell}} r(b)  \ll P^{1-\varepsilon/4}.
\end{equation}

Combining~\eqref{cp} and~\eqref{l1}, \eqref{l2} and~\eqref{l3} we have 
\begin{equation} 
\begin{split}  \label{Q1 prelim}
\boldsymbol{Q}_1(P,y) &\leq \(  \frac{1-\varepsilon}{2} +o(1) \)  L(1,\chi) P \log P \\
&\qquad\qquad\qquad +O   \(P^{1-\eta/2} +P^{1-\varepsilon/4} \).
\end{split}
\end{equation}
We recall Siegel's theorem given in Lemma~\ref{lem:Siegel}. 
Thus the main term eventually dominates the error term on 
the left side of~\eqref{Q1 prelim}. Thus, 
\begin{equation} 
 \label{Q1}
\boldsymbol{Q}_1(P,y) \leq  \frac{1}{2}  \(1-\varepsilon+o(1) \) L(1,\chi) P \log P. 
\end{equation}

\subsubsection{Estimate for $\boldsymbol{Q}_2(P,y)$}
Suppose $p \mid n$ and $\chi(p)=-1$. Then $\ord_p n$ must be even, otherwise 
$r(n)=0$ by~\eqref{quadrep}. Applying~\eqref{quadrep}, for even $\ell$ we have 
$$
r(p^{\ell}b)=r(b). 
$$
Observe that
$$
\boldsymbol{Q}_2(P,y)=\sum_{\substack{p \leq \sqrt{P} \\ \chi(p)=-1}} \log p \sum_{\ell \geq 2~\text{ even} }
\ell \sum_{\substack{ b \newsim  P/p^{\ell} \\ \gcd(b,p)=1 }} r(b).
$$

At this point we   drop the conditions 
$$
\chi(p)=-1  \mand   \gcd(b,p)=1. 
$$ 
Hence, we write
\begin{equation} 
\label{eq:Q2 simple}
\boldsymbol{Q}_2(P,y)\le \sum_{p \leq \sqrt{P}} \log p \sum_{\ell\ge 2~\text{even}}
\ell \sum_{\ b \newsim  P/p^{\ell}} r(b).
\end{equation}

For terms with   $P/p^{\ell} \geq q^{1/4+\varepsilon/3}$   replace the condition $2 \mid \ell$ with $\ell \ge 2$ and arrive at the same sum as in~\eqref{l2}, except with $y$ 
replaced by $ \sqrt{P}$ which does not change the bound (which in fact does not depend on $y$).
Hence  we now obtain
\begin{equation} 
 \label{Q2intermed}
 \sum_{p \leq \sqrt{P}}\log p  \sum_{\substack{\ell\ge 2~\text{even} \\ P/p^{\ell} \geq q^{1/4+\varepsilon/3}}}  \ell  \sum_{b \newsim P/p^{\ell} } r(b)  \\
\ll L(1,\chi) P+ P^{1-\eta}. 
\end{equation}

For terms with  $P/p^{\ell} < q^{1/4+\varepsilon/3}$, we  use~\eqref{eq: r and tau} and derive 
\begin{align*}
\sum_{p \leq \sqrt{P}} \log p  \sum_{\substack{\ell\ge 2~\text{even}\\ P/p^{\ell}<q^{1/4+\varepsilon/3}}} &  \ell \sum_{b \newsim P/p^{\ell} } r(b) \\
& \le \sum_{p \leq \sqrt{P}} \log p  \sum_{\substack{\ell\ge 2~\text{even} \\ P/p^{\ell}<q^{1/4+\varepsilon/3}}}  \ell
\sum_{b \newsim P/p^{\ell}} \tau(b). 
\end{align*}
We see from~\eqref{eq: r and tau} that inner sum over $b$ is  
bounded by 
$$
(P/p^{\ell})^{1+o(1)}  \le q^{1/4+\varepsilon/3 + o(1)}.
$$
Hence we obtain 
\begin{align*}
\sum_{p \leq \sqrt{P}}\log p  \sum_{\substack{ \ell\ge 2~\text{even} \\ P/p^{\ell}<q^{1/4+\varepsilon/3}}}  \ell \sum_{b \newsim P/p^{\ell}} r(b)
& \le  P^{1+o(1)}   \sum_{p \leq \sqrt{P}} \log p  \sum_{\substack{ \ell\ge 2~\text{even} \\ p^{\ell}> Pq^{-1/4-\varepsilon/3}}} \frac{\ell}{p^\ell}  \\
& = P^{1+o(1)}   \sum_{p  \leq \sqrt{P}}    \sum_{\substack{\ell\ge 2~\text{even} \\ p^{\ell}>  Pq^{-1/4-\varepsilon/3}}} \frac{\log(p^\ell)}{p^\ell} \\
&   \le  P^{1+o(1)}      \sum_{k>  P^{1/2} q^{-1/8-\varepsilon/6}} \frac{\log(k^2)}{k^2} \\
 &   \le  P^{1/2} q^{1/8+\varepsilon/6+o(1)}. 
\end{align*}

We see that  the condition $P \ge  q^{1/4+\varepsilon}$ implies
$$P^{1/2} q^{1/8+\varepsilon/6+o(1)} \le P q^{-\varepsilon/3 + o(1)} \ll P^{1-\varepsilon/4}
$$ \
and 
we arrive to the estimate
\begin{equation} 
 \label{l3-Q2}
\sum_{p \leq \sqrt{P}}\log p  \sum_{\substack{ \ell\ge 2~\text{even} \\ P/p^{\ell}<q^{1/4+\varepsilon/3}}}  \ell \sum_{b \newsim P/p^{\ell}} r(b))  \ll P^{1-\varepsilon/4}.
\end{equation}

Hence, combining~\eqref{Q2intermed} and~\eqref{l3-Q2} and Siegel's theorem,  given in 
 Lemma~\ref{lem:Siegel}, we  derive from~\eqref{eq:Q2 simple} that 
\begin{equation} 
 \label{Q2}
\boldsymbol{Q}_2(P,y) \ll L(1,\chi) P . 
\end{equation}

\subsubsection{Concluding the proof}
Comparing~\eqref{mainasym}, \eqref{Q1} and~\eqref{Q2},  we conclude that
\begin{equation} \label{logS3}
\boldsymbol{Q}_3(P,y) \geq \frac{1}{2}  \(\varepsilon+o(1) \) L(1,\chi)P \log P.
\end{equation}
Observe that
\begin{equation} \label{trad}
\boldsymbol{Q}_3(P,y)=\sum_{\substack{y<p \leq P  \\ \chi(p)=1}} c_p(P) \log p,
\end{equation}
where 
$$
c_p(P)= \sum_{\ell=1}^\infty  \ell(\ell+1) \sum_{\substack{ b \newsim P/p^{\ell} \\ (b,p)=1 }} r(b)  . 
$$
Estimating $r(n)$ via the divisor function as in~\eqref{eq: r and tau}, we infer that for $y<p \leq P$, 
\begin{equation} \label{upper}
c_p(P) \log p \le  P y^{-1} q^{o(1) }.
\end{equation}

Comparing~\eqref{logS3}, \eqref{trad} and~\eqref{upper},
we see  that the sum $\boldsymbol{Q}_3(P,y)$ is supported on at least 
$$
\boldsymbol{Q}_3(P,y)  P^{-1}   y  q^{o(1)}\ge \frac{1}{2} \bigl(\varepsilon+o(1) \bigr)  L(1,\chi) y  q^{o(1)}
$$
split primes between $y$ and $P$.  Choosing the largest $y$ that satisfies 
both~\eqref{con1} and~\eqref{con2}, 
that is
$$
y = \min\left\{ P  q^{-1/4-2\varepsilon/3},  P^{1/2-\varepsilon/2} \right\}
$$  
as well as applying Siegel's theorem given in Lemma~\ref{lem:Siegel}, 
concludes the proof of Theorem~\ref{uncondthm1}.

\subsection{Effective lower bounds: Proof of Theorem~\ref{uncondthm2}} \label{sec:uncondthm2}
Recall that now we assume $q \equiv 3 \pmod {16}$. Let 
$$
\boldsymbol{Q}(X,Y)=X^2+\frac{q+1}{4} Y^2
$$
be the principal form of discriminant $-q$.
We define
$$
\boldsymbol{P}(X):=\boldsymbol{Q}(X,1),
$$
and consider   the product
$$
\boldsymbol{R}_q:=\prod_{1 \leq n \le t} \boldsymbol{P}(n),
$$
where 
$$
t = \fl{\sqrt{3q}/2}.
$$
Observe that that for $1 \le n \le t$ we have 
$$
q/4 < \boldsymbol{P}(n) \le q,
$$
and 
\begin{equation} \label{sizeR}
 (q/4)^t <\boldsymbol{R}_q<q^t.
\end{equation}

Recall that an integer $k$ is represented by some quadratic form of discriminant $-q$ if and 
only if there is a solution $b$ to $b^2 \equiv -q \pmod{4k}$, see~\cite[Equation~(22.21)]{IwKow}. Thus 
if $p \mid k$, then $b^2 \equiv -q \pmod{p}$ has a solution and so $p$ must be split. Thus the prime factorisation
of $\boldsymbol{R}_q$ contains only split primes. Hence our strategy  is to obtain a lower bound for
\begin{equation} \label{eq: omega Nq}
\omega(\boldsymbol{R}_q) \le N_q,
\end{equation}
the number of distinct prime factors of $\boldsymbol{R}_q$.

Clearly,  for each prime $p\ge 3$ dividing  $\boldsymbol{R}_q$, the congruence 
\begin{equation} 
 \label{eq: mod p}
\boldsymbol{P}(x)\equiv 0 \pmod{p}
\end{equation} 
has  two distinct non-zero  solutions in  $\mathbb F_p$ which are not roots of the derivative $\boldsymbol{P}^{\prime}(X) = 2X$.
Hence, applying  {\it Hensel's lifting\/}  these roots can be uniquely lifted to $p$-adic solutions $\alpha_p, \beta_p\in \Z_p$. 
Let $\kappa_p \in \mathbb{N}$ be the least positive integer such that $p^{\kappa_p+1}> t$  and
 let $1 \leq  a_p < b_p <p^{\kappa_p}$
be the unique integers such that 
 \begin{equation} 
 \label{eq: ord large}
\ord_p\(\alpha_p- a_p\), \ord_p\( \beta_p- b_p\)  \ge \kappa_,
\end{equation} 
where, as before,  $\ord_p n$ denotes the $p$-adic order of $n \in \Z$. 
For $1 \leq n\le t$ and $n \neq  a_p ,  b_p$ we trivially have 
 \begin{equation} 
 \label{eq: ord small}
\ord_p ( n- a_p), \ord_p (n- b_p) < \kappa_p  .
\end{equation} 
Using~\eqref{eq: ord large} and~\eqref{eq: ord small}, we derive for such $n$
 \begin{equation} 
 \begin{split}
 \label{ordp}
\ord_p \boldsymbol{P}(n)&=\ord_p ( n-\alpha_p )+\ord_p (n-\beta_p)    \\
&=\ord_p ( n- a_p- \alpha_p+a_p)+\ord_p (n- b_p -  \beta_p +  b_p)  \\
&=\ord_p ( n- a_p)+\ord_p (n- b_p).
 \end{split}
\end{equation} 

We observe that $\ord_p\( a_p - b_p\) = 0$ as the underlying solutions
to the congruence~\eqref{eq: mod p} are distinct modulo $p$. 
This  ensures that for any given $n \in \mathbb{N}$, both terms on the right side of~\eqref{ordp} 
cannot be simultaneously non-zero. Thus
\begin{align*}
\ord_p \boldsymbol{R}_q= \ord_p\boldsymbol{P}( a_p) & +\sum_{\substack{1 \leq n \le t\\  n \neq  a_p  }} \ord_p \( n- a_p \)\\
& +\ord_p\boldsymbol{P}( b_p) +\sum_{\substack{1 \leq n \le t  \\  n \neq  b_p  }} \ord_p \( n- b_p \).
\end{align*}
We now write
\begin{align*}
\sum_{\substack{1 \leq n \le t\\  n \neq  a_p  }} \ord_p \( n- a_p \) & = 
\sum_{1 \leq n <  a_p}  \ord_p \( n- a_p \) + \sum_{ a_p + 1 \leq n \le  t} \ord_p \( n- a_p \) \\
& = 
\sum_{1 \leq n <  a_p}  \ord_p \( n\) + \sum_{1 \leq n \le  t-a_p} \ord_p \( n\) \\
& \le  \ord_p    (a_p-1) !  +  \ord_p \(t - a_p\) ! \le   \ord_p  (t-1)!
\end{align*}
and similarly for the other sum. 
Hence, using the trivial bound 
$$
p^{\ord_p \boldsymbol{P}( a_p)}   \le \boldsymbol{P}( a_p)   \le \boldsymbol{P}(p^{\kappa_p}) \le    \boldsymbol{P}(t)
= t^2 + \frac{q+1}{4}   \le q
$$
we obtain 
\begin{equation}  \label{ordpRq}
\begin{split} 
  \ord_p \boldsymbol{R}_q  & \le   \ord_p \(\boldsymbol{P}( a_p)\boldsymbol{P}( b_p) \((t-1) ! \)^2 \) \\
  & \le 2 \frac{\log q}{\log p} 
  +   2  \ord_p  (t -1)! .
  \end{split} 
\end{equation}
 
Since $q \equiv 3 \pmod{16}$, we have 
$$ 
\ord_2\boldsymbol{P}(n)=\begin{cases}
 0 \quad & \text{if}  \quad n  \text{ is even},\\
1 \quad & \text{if}  \quad n \text{ is odd},
\end{cases}
$$
thus 
\begin{equation} \label{ord2ineq} 
\ord_2 \boldsymbol{R}_q \le t/2.
\end{equation}
Combining~\eqref{ordpRq} and~\eqref{ord2ineq} to bases $p$ and $2$ respectively, and then taking the product over all primes, 
we derive
\begin{equation} \label{Rq}
\boldsymbol{R}_q<2^{t/2}q^{2 \omega(\boldsymbol{R}_q)} \((t-1)! \)^2.
\end{equation}
Recalling the lower bound from~\eqref{sizeR}, then~\eqref{Rq} implies that
\begin{equation} \label{logprep}
(q/4)^{t}<2^{t/2} q^{2 \omega(\boldsymbol{R}_q)} \((t-1)! \)^2.
\end{equation}  
By the Stirling  formula, we have
\begin{equation}\label{stirling}
(t-1)! \le  (t-1)^{t-1/2} e^{-t+2} \le (t/e)^t
\end{equation}
provided 
$t > 7$ which holds for $q \ge 67$.

Taking the logarithm of both sides of~\eqref{logprep} and using~\eqref{stirling} yields  
$$
\log q-\log 4<  \frac{1}{2} \log 2+ \frac{2 \omega(\boldsymbol{R}_q) \log q}{t}+ 2 \log t   -2.
$$  

Since 
$$
\log q -  2 \log t =  \log q -    \log t^2  \ge  \log q -    \log(3q/4)    = -     \log(3/4)  
$$
and
$$
2-  \log 4 -    \frac{1}{2} \log 2 -   \log(3/4) = 2-\log (3\sqrt{2})
$$
we see that 
$$
 \omega(\boldsymbol{R}_q)\ge \frac{\(2-\log (3\sqrt{2}) \) }{2}
  \frac{t}{\log q}  . 
$$
Recalling~\eqref{eq: omega Nq} we conclude the proof of Theorem~\ref{uncondthm2}.

\section{Counting  split primes conditionally: Proof of Theorem~\ref{siethm}} \label{sec:siethm}

\subsection{Values of quadratic forms free of small prime divisors} 
Suppose 
$$
f(U,V)=AU^2+BUV+CV^2 \in \Z[U,V],
$$ 
is a binary quadratic form with discriminant 
$$B^2 - 4AC = -q\equiv 1 \pmod 4;
$$ 
we refer to~\cite[Section~22.1]{IwKow} for a general background.
  Suppose the form is reduced, that is,
\begin{equation} \label{reducedco1}
\gcd(A,B,C)=1 \quad \text{and} \quad |B| \leq A \leq C.
\end{equation}
This implies that $3AC \le 4AC - B^2 = q$. Hence 
\begin{equation} \label{reducedco2}
A \leq \sqrt{\frac{q}{3}} \quad \text{and} \quad AC \leq \frac{q}{3}.
\end{equation}

Let $P \geq 1$ be any number such that $P \geq C q^{\varepsilon}$. For each 
$$
1 \leq v \leq \sqrt{P/C} \quad \text{with} \quad \gcd(v,A)=1,
$$
consider 
$$
F_v(U):=f(U,v)=AU^2+BUv+Cv^2  \in \Z[U].
$$
We still have 
\begin{equation} \label{eq:gcd=1}
\gcd(A,Bv,Cv^2)=1. 
\end{equation}
Thus the reduction of $F_v \in \Z[U]$ modulo each odd prime $p$ is non-zero. Only $p=2$ could be a common factor
of all values of $F_v$. 

Now, suppose 
\begin{equation} \label{eq:3 div}
4 \mid F_v(0), \qquad 4 \mid F_v(1), \qquad 4 \mid F_v(2)
\end{equation}
or equivalently
$$
4 \mid Cv^2, \qquad 4 \mid A+Bv+Cv^2, \qquad 4 \mid 4A+2Bv+Cv^2. 
$$
From the first and the third divisibility we conclude that $2 \mid Bv$. Since we have~\eqref{eq:gcd=1}, we see that $A$ must be odd and so 
$F_v(1)=A+Bv+Cv^2$ is odd. Thus the divisibilities~\eqref{eq:3 div} are not possible simultaneously. 
Hence one can always choose $n_0 \in \{0,1,2\}$ such that $4 \nmid F_v(n_0)$. Let $e_v \in \{0,1\}$ be such that $2^{e_v} \mid F_v(n_0)$ but 
 $2^{e_v+1} \nmid F_v(n_0)$.   Then the polynomial 
$$
G_v(U)=\frac{1}{2^{e_v}} F_v(4U+n_0)
$$ 
produces only odd values that are also the largest odd divisors of values of  the   polynomial $F_v(4U+n_0)$. 
Hence it is enough to show that there are ``many'' values of $G_v$ with at most 5 prime divisors. 

Define the sequence 
$$
\cA_v:=\{ G_v(n)\}_{1 \leq n \leq \sqrt{P/A}}.
$$
All the elements in $\cA_v$ are less than an absolute constant times $P$. Observe that 
$$
G_v(x) \equiv 0 \pmod{p}
$$
has no solutions when $p=2$ by construction and at most two solutions when $p$ is odd by Lagrange's theorem. Thus we have 
$$
\# \cA_v(p) =g_v(p)  \# \cA_v+r_{v}(p)  
$$
where 
\begin{equation} \label{eq:gvp}
0 \leq g_v(p) \leq 2/p, \quad |r_{v}(p)| \leq 2,
\end{equation}
and $\cA_v(p)$ denotes elements of the sequence $\cA_v$ which are divisible by $p$. We can extend $g_v(p)$ to a multiplicative function $g_v(d) $ supported on squarefree $d$ and using  the Chinese Remainder Theorem, we obtain 
\begin{equation} \label{eq:rvd}
\# \cA_v(d)=g_v(d) \# \cA_v+r_{v}(d) \mand |r_{v}(p)| \leq 2^{\omega(d)}
\end{equation}
with 
$$
g_v(d) := \prod_{\substack{p\mid d\\
p~\text{prime}}} g_v(p) \mand \omega(d) := \sum_{\substack{p\mid d\\
p~\text{prime}}} 1. 
$$

We now need a lower bound on the number of elements in the sequence $\cA_v$ which are free of 
small prime divisors.  
To do this we appeal to~\cite[Theorem~11.13]{FrIw} and thus below we 
try to match the notation from~\cite{FrIw}.  Namely, we are interested in a good  lower bound on 
cardinality of the set 
$$
\cS(\cA_v, z) := \{n \in [1, \sqrt{P/A}]:~ p \mid G_v(n) \Longrightarrow p \ge z \}. 
$$
From~\eqref{eq:gvp} and the Mertens formula, see~\cite[Proposition~2.2]{FrIw},
we see that for any $2 \le w < z$ we have
$$
\prod_{w\le p <z} \(1-g_v(p)\)^{-1} <  \(\frac{\log z}{\log w}\)^2 \(1 + O\( (\log w)^{-1}\)\).
$$
Hence~\cite[Equation~11.129]{FrIw}  is satisfied with $\kappa = 2$ and thus by the table in~~\cite[Section~11.19]{FrIw}  
 we can take $\beta  = 4.833986 \ldots$ in~\cite[Theorem~11.13]{FrIw}. We conclude that for any fixed $s \ge \beta$ 
 there is a constant $c(s)>0$ such that  with $D = z^s$ we have
 $$
\#  \cS(\cA_v, z)  \ge  \(c(s)+o(1)\) V_v(z) X + R(D,z), 
 $$
 where 
 $$
X := \sqrt{P/A},  \quad V_v(z) := \prod_{p <z} \(1-g_v(p)\), \quad R(D,z) := \sum_{\substack{d < D\\p\mid d\, \Rightarrow\,  p < z}}|r_{v}(d)|. 
$$
From the bound $2^{\omega(d)} \leq \tau(d)$, the equation~\eqref{eq:rvd} and the bound on the divisor function~\cite[Equation~(1.81)]{IwKow},
we have $ |r_{v}(d)|  = d^{o(1)}$. Thus we have 
a trivial bound 
$$
 R(D,z)  \le D^{1+o(1)} = z^{s + o(1)}. 
$$
Also by the Mertens formula, see~\cite[Proposition~2.2]{FrIw}, we have
 $$
V_v(z) \gg \frac{1}{(\log z)^2}. 
$$
Taking $s = 4.85> \beta$ and $z = X^{10/49}$,  we conclude that 
\begin{equation} \label{eq:Sieve-LB}
\#  \cS(\cA_v, X^{10/49})  \gg   \frac{X}{(\log X)^2}. 
\end{equation}
Hence $\cA_v$ contains at least $c_0 \sqrt{P/A}/(\log q)^2$ elements with at most 5 prime factors, for some absolute 
constant $c_0 > 0$. Note that the choice $P$ ensures that 
$$X \ge  \sqrt{P/A} \ge   \sqrt{P/C} = q^{\varepsilon/2}. 
$$
Thus 
  the right side of~\eqref{eq:Sieve-LB} is meaningful.  

\subsection{Construction of the set $\mathcal{F}_{-q}$} 

We start with emphasising that the implied constant  in~\eqref{eq:Sieve-LB} is absolute. 

Consider a complete set of $h(-q)$ inequivalent forms for the class group of discriminant $-q$. Denote them 
\begin{equation} \label{binquad}
f_t(U,V)=A_t U^2 +B_t UV +C_t V^2 \in \Z[U,V], \quad t=1,\ldots,h(-q), 
\end{equation}
where $h(-q)$ is the class number, see~\cite[Section~22.2]{IwKow}.

The primitive binary quadratic forms in~\eqref{binquad} are in bijection with the set of Heegner points 
$$
\Lambda_{-q}:=  \left \{  z_{Q_t}:=\frac{-B_t+\sqrt{-q}}{2A_t} :~B_t^2-4 A_t C_t=-q, \ z_{Q_t} \in \mathcal{D}  \right\},
$$
where $\mathcal{D}$ is the standard fundamental domain for the modular group ${\mathrm SL}_2(\Z)$, see~\cite[Section~14.1]{IwKow}
or~\cite[Section~1.2]{SA}. 
Consider 
$$
\Omega:=  \left \{ \tau=x+iy:~-1/2 \leq x \leq 1/2, \quad 1 \leq y \leq 10 \right\} \subseteq \mathcal{D}.
$$
 The equidistribution theorem of Duke~\cite[Theorem~1]{Duk} yields 
\begin{equation}  \label{dukedis}
\frac{\#\( \Lambda_{-q} \cap \Omega\)}{ \# \Lambda_{-q}}= \frac{27}{10 \pi}+O(q^{-\delta}), \quad \text{as} \quad q \rightarrow \infty,
\end{equation}
where $\delta>0$ depends only on $\Omega$, and the implied constant depends only on $\Omega$ and $\delta>0$, but is ineffective. 
Note that~\eqref{dukedis} is taken with respect to the normalised hyperbolic area measure 
$$
d \mu(\tau)=\frac{3}{\pi} dx dy/y^2, 
$$
where as in the above $\tau=x+iy$.

Suppose the forms in~\eqref{binquad} corresponding to to Heegner points in $\Omega$ are indexed by $t \in \mathcal{S}_{-q}$. Thus~\eqref{dukedis} guarantees
\begin{equation} \label{Slow}
\# \mathcal{S}_{-q} \geq \(\frac{27}{10 \pi}+o(1) \) h(-q). 
\end{equation}
The binary quadratic forms with $t \in \mathcal{S}_{-q}$ satisfy~\eqref{reducedco1}, \eqref{reducedco2} as well as
$$
 1 \leq   \Im z_{Q_t} = \frac{\sqrt{q}}{2A_t} \leq 10. 
$$
Thus 
\begin{equation} \label{coeffsq}
\max(|A_t|,|B_t|,|C_t|) \leq \frac{20}{3} \sqrt{q}, \qquad t \in \mathcal{S}_{-q}.
\end{equation}

Now we restrict our attention to  $\{f_t\}_{t \in \mathcal{S}_{-q}}$ in~\eqref{binquad}. For each such $f_t$, one can form the polynomials $G_u^{(t)}$ and sequences $\cA_u^{(t)}$ as above. Thus by~\eqref{eq:Sieve-LB} we have 
\begin{equation} \label{eq:Sieve-LB t}
\# \cS^\sharp(\cA_u^{(t)}, X_t^{10/49}) \gg \frac{X_t}{(\log X_t)^2}, \qquad t \in \mathcal{S}_{-q}. 
\end{equation}
with implied constant independent of $t$ and $u$, where $X_t := \sqrt{P/A_t}$ and $\cS^\sharp$ indicates that only square-free elements of a set $\cS\subseteq \Z$ are included. 
 
Let 
$$
\cF_{-q}:= \bigcup_{t \in \mathcal{S}_{-q}} \bigcup_{u=1}^{\sqrt{P/C_t }}  \cS^\sharp(\cA_u^{(t)}, X_t^{10/49}) .
$$
We now need a lower bound for $\# \cF_{-q}$. Recall that by our construction, all elements of $\cF_{-q}$ are odd. Since each $n \in\cF_{-q}$ has at most $5$ prime factors, we see that~\eqref{quadrep} implies 
$$
0 < R_{-q}(n)\le 2^6.
$$
Thus, recalling~\eqref{Slow}, \eqref{coeffsq} and~\eqref{eq:Sieve-LB t},
 we derive
\begin{equation} \label{Fqbd}
\#\cF_{-q}  \gg \sum_{t \in \mathcal{S}_{-q}} \frac{P}{\sqrt{A_t C_t} \log^2 q} \gg  h(-q) \frac{P}{\sqrt{q} \log^2 q}.
\end{equation}

 \subsection{Descent to split primes}  \label{descent} 
We recall an integer $n$ is represented by some quadratic form of discriminant $-q$ if and 
only if there is a solution $b$ to $b^2 \equiv -q \pmod{4n}$, see~\cite[Equation~(22.21)]{IwKow}. Thus 
if $p \mid n$, then $b^2 \equiv -q \pmod{p}$ has a solution and so $p$ must be split. Each $n \in \mathcal{F}_{-q}$
is squarefree and odd, so if $p \mid n$, then $n=pm$ where $p$ is split and $m$ is represented by a quadratic form of discriminant $-q$.

Now we come to the heart of the argument, where we sift $\cF_{-q}$ down to just primes using the Siegel zero. Cover
the interval $[P^{10/49},P^{1/2}]$ into $O(\log q)$ dyadic intervals $[A,2A]$ and let $\nu \in \{2,3\}$ be the least integer such that $A^{\nu}>P^{1/2}$. 

 Let $\cQ_{-q}(A)$ denote the set of  primes  $p\in [A,2A]$ that are split. 
 Our goal is to show that the number of elements  $n \in\cF_{-q}$ that are divisible 
 by some $p \in \cQ_{-q}(A)$ is significantly less than the lower bound for  $\#\cF_{-q}$ established in~\eqref{Fqbd}. The remaining elements of $\cF_{-q}$ are primes. 

To show this, we denote  
 $$
Q_{-q}(A) = \# \cQ_{-q}(A). 
$$
Consider all square-free products of $\nu$  such primes. There are 
\begin{equation} \label{eq:Q nu}
\binom{Q_{-q}(A)}{\nu}  \gg   Q_{-q}(A)^{\nu}
\end{equation} 
such products.  If $n$ is a product of split primes then $x^2 \equiv -q \pmod{n}$ is solvable, and since $q \equiv 3 \pmod{4}$, we can lift the congruence to 
$x^2 \equiv -q \pmod{4n}$, so $R_{-q}(n) \geq 1$. 

In what follows all  implied constants may depend on $\varepsilon> 0$ (but are ineffective). 

Since $A^{\nu}>P^{1/2}>q^{1/4+\varepsilon/2}$, 
the asymptotic formula~\eqref{quadpol}  and the inequality~\eqref{eq:Q nu} imply  that 
\begin{equation} \label{QAineq}
Q_{-q}(A) \ll  L(1,\chi)^{1/\nu} A+O(A^{1-\eta}).
\end{equation}

We now recal  Siegel's theorem, see Lemma~\ref{lem:Siegel}. 
Since  $A^\nu  \ge q^{1/4}$, taking $\delta = \eta/8$ in Lemma~\ref{lem:Siegel}  we obtain 
$$
L(1,\chi) \ge  C(\eta/8) q^{-\eta/8} \ge C(\eta/8)  A^{-\nu \eta/2}.
$$ 
Thus, if $q$ is large enough the first  term on the right hand side of~\eqref{QAineq} dominates 
and we derive
$$  
Q_{-q}(A) \ll   L(1,\chi)^{1/\nu} A.
$$
Thus by assumption~\eqref{Lsize} we have 
\begin{equation} \label{siegelimp}
\frac{Q_{-q}(A)}{A} \ll \frac{1}{(\log q)^{10/3}}.
\end{equation}

By the discussion at the start of Section~\ref{descent}, if $p \in \cQ_{-q}(A)$ divides $n \in \cF_{-q}$, then $n=pm$ where $R_{-q}(m) \geq 1$. Thus the number of elements of $\cF_{-q}$  that are divisible by a $p \in \cQ_{-q}(A)$ is bounded from above by 
$$
\sum_{m \leq P/p} R_{-q}(m).
$$
 Since $P/p \geq q^{1/4+\varepsilon/2}$, the asymptotic formula~\eqref{quadpol} implies that 
 $$
 \sum_{m \leq P/p} R_{-q}(m)=2 L(1,\chi) \frac{P}{p}+O  \( \( Pp^{-1} \)^{1-\eta} \) \ll L(1,\chi) \frac{P}{p},
 $$
  where the last inequality follows from  Lemma~\ref{lem:Siegel} (we recall that the implied 
 constants may depend on $\varepsilon$).  
 Summing this contribution over all $p \in Q_{-q}(A)$, using~\eqref{siegelimp} and Dirichlet's class number formula yields 
 \begin{align*}
\#\{n \in\cF_{-q}:~ &  \exists \, p  \in \cQ_{-q}(A)\ \text{such that}\ p\mid n\} \\ 
& \ll L(1,\chi) P  \sum_{p \in Q_{-q}(A)} \frac{1}{p}  \ll L(1,\chi) P \frac{Q_{-q}(A)}{A} \\&
 \ll \frac{h(-q) P}{\sqrt{q} (\log q)^{10/3}}. 
 \end{align*}
Summing the last display over each of the $O(\log q)$ dyadic intervals $[A,2A]$, 
we see that the number of composite integers in  $\cF_{-q}$ 
is  
$$
O\(h(-q) P/ \sqrt{q} (\log q)^{7/3}\).
$$
Comparing this with~\eqref{Fqbd}  we conclude the proof of Theorem~\ref{siethm}.
 
\section{Bounds of  bilinear Weyl sum with square roots: Proof of Theorem~\ref{thm:Waq}} 

\subsection{Geometry of numbers and congruences}
 The following is Minkowski's second theorem, for a proof see~\cite[Theorem~3.30]{TaoVu}.
 
\begin{lemma}
\label{lem:mst}
Suppose $\Gamma \subseteq \R^{d}$ is a lattice of determinant $\det \Gamma$,  $\sfB\subseteq \R^{d}$ a symmetric convex 
body of volume $\Vol(\sfB)$ and let $\lambda_1,\ldots,\lambda_d$ denote the successive minima of $\Gamma$ with respect to $\sfB$. Then we have
$$\frac{1}{\lambda_1\ldots\lambda_d}\le \frac{d!}{2^d}\frac{\Vol(\sfB)} {\det \Gamma}.
$$
\end{lemma}

A proof of the following is given in~\cite[Proposition~2.1]{BHM}. 

\begin{lemma}
\label{lem:latticesm}
Suppose $\Gamma \subseteq \R^{d}$ is a lattice, $\sfB\subseteq \R^{d}$ a symmetric convex body and let $\lambda_1,\ldots,\lambda_d$ denote the successive minima of $\Gamma$ with respect to $\sfB$. Then we have
$$|\Gamma \cap \sfB|\le \prod_{j=1}^{d}\left(\frac{2j}{\lambda_j}+1\right).$$
\end{lemma}

Using Lemmas~\ref{lem:mst} and~\ref{lem:latticesm} we give a variant of a result due to Bourgain, Garaev, Konyagin and Shparlinski~\cite[Lemma~5]{BGKS}. Note the main difference between our result and~\cite[Lemma~5]{BGKS} is an alternate treatment of the case with long intervals.

\begin{lemma}
\label{lem:shortpoint}
Let $q$ be prime  and let $\cI,\cJ$  be two intervals containing $h$ and $H$ integers, respectively.
For an integer $s$, let $I(s)$ count the number of solutions to  the congruence 
\begin{equation} 
\label{eq:maincong}
x\equiv ys \pmod{q}, \qquad x\in \cI, \ y\in \cJ.
\end{equation}
Then either 
$$
I(s)\ll \max\left\{\frac{Hh}{q},1\right\},
$$
or  there exist  $a,b \in \Z$  with 
\begin{equation}
\label{eq:case12}
  |a|\ll \frac{h}{I(s)} \mand |b|\ll \frac{H}{I(s)},
\end{equation}
satisfying
\begin{equation}
\label{eq:case13}
 s\equiv ba^{-1} \pmod{q}.
\end{equation}
\end{lemma}

\begin{proof}  
If $I(s)\ll 1$ then there is nothing to prove, so we assume that $I(s) \ge C_0$ for some 
sufficiently large absolute constant $C_0$.  Then there exists some $x_0\in \cI,y_0\in \cJ$ such that 
$$
x_0\equiv y_0s \pmod{q}.
$$
Hence for any other $x,y$ satisfying~\eqref{eq:maincong} we have 
$$
(x-x_0)\equiv (y-y_0)s \pmod{q}.
$$
Define the lattice $\cL$ and the box $B$ by 
$$
\cL:=\{ (x,y)\in \Z^2 :~x\equiv ys \pmod{q}\},
$$
and
$$
 B:=\{ (x,y)\in \R^2 :~|x|\le h, \ \ |y|\le H\},
$$
respectively, 
so that 
$$
I(s)\le  \#\(\cL\cap B\) .
$$

Let $\lambda_1 \le \lambda_2$ denote the successive minima of $\cL$ with respect to $B$. 
Clearly $\cL$ is of determinant $\det \cL = q$.  

If $\lambda_2 \le 1$ then by  Lemmas~\ref{lem:mst} and~\ref{lem:latticesm} we have
$$
I(s)\ll \frac{1}{\lambda_1\lambda_2}\ll \frac{Hh}{\det \cL}= \frac{Hh}{q}.
$$ 

Hence  we may now assume $\lambda_2>1$. 
Then  by Lemma~\ref{lem:latticesm}
$$
I(s)  \ll \frac{1}{\lambda_1} + 1.
$$
Since we assume $I(s)\ge C_0$, for a sufficiently large $C_0$ this implies 
$$
\lambda_1 \ll I(s)^{-1}, 
$$
and hence for suitable constant $c$ that 
$$
\cL\bigcap \frac{c}{ I(s)}\cdot B\neq \{0\}, 
$$
(where $\lambda \cdot B$ means homothetic scaling of $B$).  
From this it follows that there exists $a$ and $b$ satisfying~\eqref{eq:case12} and~\eqref{eq:case13}. 
\end{proof}

\subsection{Additive energy of modular square roots}  
\label{sec:E-bounds}
Our argument is based on a weighted additive energy for modular square roots, which is of independent 
interest. 

For a complex weight $\bbeta$ as in~\eqref{eq:weight} and $j \in \Fq^{\times}$ we define the {\it weighted additive energy} 
$$
E_{q,j}(\bbeta):=\sum_{\substack{(u,v,x,y) \in \Fq^4 \\ u+y=x+v }} \beta_{ju^2} \overline{\beta}_{jv^2} \beta_{jx^2} \overline{\beta}_{jy^2}.
$$
Recall that $ju^2,jv^2, jx^2, jy^2$ are  all computed modulo $q$ and take the value of the reduced residue between $1$ and $q$. We omit the subscript $j$ when $j=1$.
Quantities of this type are well known in additive combinatorics under the name of {\it additive energy\/}. 

It is also convenient to define 
 \begin{equation} \label{Qdef}
 Q_{\lambda,j}(\bbeta):= \sum_{\substack{ (u,v) \in \Fq^2 \\ u - v = \lambda }} \beta_{ju^2} \overline{\beta}_{jv^2} . 
 \end{equation}
 and observe that 
 \begin{equation} \label{eq:E Q2}
E_{q,j}(\bbeta) = \sum_{\lambda \in\Fq}  Q^2_{\lambda,j}(\bbeta).
 \end{equation}

In particular, following our convention that $\mathbf{1}_{[N,2N]}$ denotes the characteristic function of the interval $[N,2N]$  we see that 
$$ 
Q_{\lambda,j}\(\mathbf{1}_{[N,2N]}\)= \# \{ \(u,v\) \in \Fq^2: ~ju^2, jv^2 \in [N,2N], \ u-v=\lambda\}
$$
where all arithmetic operations are performed in $\Fq$.

\begin{lemma}
\label{lem:MSR}
For $j\in \Fq^{*}$ and an  integer $N\le q/2$, there is an absolute  constant $c > 0$ such that 
for any  $\lambda \in  \Fq^{*}$    either 
$$
Q_{\lambda,j}\(\mathbf{1}_{[N,2N]}\) \ll  \max \left\{\frac{N^{3/2}}{q^{1/2}},1\right\},
$$
or  there exist  $a,b\in \Z$ with 
$$
|a|\le \frac{N^{2+o(1)}}{Q^2_{\lambda,j}\(\mathbf{1}_{[N,2N]}\)} \mand  |b|\le \frac{N^{1+o(1)}}{Q_{\lambda,j}^2\(\mathbf{1}_{[N,2N]}\)}, 
$$
satisfying
$$ 2j\lambda ^2\equiv ab^{-1} \pmod{q}.
$$
\end{lemma} 

\begin{proof}
Suppose that  $ju^2, jv^2 \in [N,2N]$ satisfy $u-v\equiv \lambda$. We now see 
that 
$$
u^2-v^2 - \lambda^2 \equiv  \lambda( u + v) - \lambda^2  \equiv   \lambda( \lambda  + 2 v) - \lambda^2 \equiv  2 \lambda v \pmod{q} .
$$
Squaring and multiplying by $j^2$, we obtain 
$$
\left( (ju^2-jv^2)-j\lambda^2 \right)^2 \equiv  4j^2\lambda^2v^2_2 \pmod{q} .
$$
 Making the substitution 
$$ju^2-jv^2\rightarrow n, \quad jv^2\rightarrow m,$$
we derive 
\begin{equation}
\begin{split}
\label{eq:KE778}
Q_{\lambda,j}&\(\mathbf{1}_{[N,2N]}\)\\& \le \# \{ |n|,|m|\le  4N,  \ (n-j\lambda^2)^2 \equiv 4j\lambda^2m \pmod{q}  \},
\end{split} 
\end{equation} 
where we now consider variables belonging to $\Z$. If $n,m$ satisfy 
$$
(n-j\lambda^2)^2  \equiv 4j\lambda^2m \pmod{q},
$$
then 
$$
n^2+j^2 \lambda^4 \equiv  2j\lambda^2(2m+n) \pmod{q}.
$$
Using~\eqref{eq:KE778} and another change of variables $2m+n \rightarrow m$ gives
\begin{equation}
\label{eq:IDb11}
Q_{\lambda,j}\(\mathbf{1}_{[N,2N]}\) \le \# \{  |n|,|m|\le 12N :~n^2+j^2d^4\equiv  4j\lambda^2m \pmod{q} \}.
\end{equation}
For each $|n|\le 12N$ there exists at most one value of $m$ satisfying the congruence in~\eqref{eq:IDb11} and for any two such pairs 
$(n_1,m_1), (n_2,m_2)$  we have 
$$
(n_1^2+n_2^2)+2j^2 \lambda^4\equiv  2j\lambda^2(m_1+m_2) \pmod{q}.
$$

This implies
\begin{align*}
Q^2_{\lambda,j}\(\mathbf{1}_{[N,2N]}\) \le \# \{ |n_1|,|n_2|, |m|&\le 12 N:\\
&n_1^2+n_2^2+2j^2 \lambda^4\equiv  2j\lambda^2m\pmod{q}\}.
\end{align*}
It is well known for any integer $r\ge 0$
$$
\# \{ n_1,n_2\in \Z :~n_1^2+n_2^2=r \}=r^{o(1)},
$$
see~\cite[Equations~(1.51) and~(1.81)]{IwKow}, 
which gives
\begin{align*}
Q^2_{\lambda,j}\(\mathbf{1}_{[N,2N]}\) \le \#\{ |n|\le 288N^2, & \ |m|\le 12 N :\\
&n+2j^2 \lambda^4 \equiv 2j\lambda^2m \pmod{q}\} N^{o(1)}. 
\end{align*}
We now apply Lemma~\ref{lem:shortpoint} with $\cI$ an interval of length $h=576N^2$ and $\cJ$ an interval of length $H=24N$, 
which completes the proof.
\end{proof}

Next we estimate the fourth moments of the quantities $Q_{\lambda,j}\(\mathbf{1}_{[N,2N]}\)$. 

\begin{lemma}
\label{lem:EB}
For $j\in\Fq^{*}$ and integer $N\le q/2$ we  have 
$$
 \sum_{\lambda \in\Fq^{*}}  Q^4_{\lambda,j}\(\mathbf{1}_{[N,2N]}\)\le \frac{N^{13/2+o(1)}}{q^{3/2}}+N^{3+o(1)}.
$$
\end{lemma}

\begin{proof}
By the dyadic Dirichlet pigeonhole principle there exists some $\Delta > 0$ such that for 
$$
\Lambda:=\{\lambda \in\Fq^{*} :~ \Delta \le Q_{\lambda,j}\(\mathbf{1}_{[N,2N]}\)<2\Delta\},
$$
we have 
\begin{equation}
\label{eq:XXXv}
 \sum_{\lambda \in\Fq^{*}}  Q^4_{\lambda,j}\(\mathbf{1}_{[N,2N]}\)\le N^{o(1)}\Delta^4 \# \Lambda.
\end{equation}
We also have trivial inequalities
\begin{equation} \label{trivineq}
\Delta\#  \Lambda \ll \sum_{ \lambda \in  \Lambda }  Q_{\lambda,j}\(\mathbf{1}_{[N,2N]}\)\le
 \sum_{ \lambda \in  \Fq}  Q_{\lambda,j}\(\mathbf{1}_{[N,2N]}\) \ll N^2.
\end{equation}
We now fix some absolute constant $c > 0$ and consider the following  two cases 
\begin{equation}
\label{eq:Ecase1}
\Delta \le c  \left(\frac{N^{3/2}}{q^{1/2}}+1\right),
\end{equation}
or 
\begin{equation} \label{eq:Ecase2}
\Delta > c \left(\frac{N^{3/2}}{q^{1/2}}+1\right).
\end{equation}

In the case when~\eqref{eq:Ecase1} holds, we have 
\begin{equation}
\label{eq:CASE177}
\begin{split}
\Delta^4\# \Lambda &\ll \left(\frac{N^{9/2}}{q^{3/2}}+1\right)\Delta\#  \Lambda
\\& \ll \left(\frac{N^{9/2}}{q^{3/2}}+1\right)\sum_{ \lambda \in  \Lambda }  Q_{\lambda,j}\(\mathbf{1}_{[N,2N]}\)\ll \frac{N^{13/2}}{q^{3/2}}+N^2,
\end{split} 
\end{equation}
where the last inequality follows from \eqref{trivineq}.

In the case when~\eqref{eq:Ecase2} holds,
we apply Lemma~\ref{lem:MSR}. 
For each $\lambda \in  \Lambda$, there exists $a,b \in \Fq^*$ satisfying 
$$
|a|\le \frac{N^{2+o(1)}}{\Delta^2}, \qquad |b|\le \frac{N^{1+o(1)}}{\Delta^2}, \qquad 2j\lambda ^2\equiv ab^{-1} \pmod{q}.
$$
Since the ratio $ab^{-1}$ can take at most 
$$
\frac{N^{2+o(1)}}{\Delta^2} \cdot  \frac{N^{1+o(1)}}{\Delta^2}  =  \frac{N^{3+o(1)}}{\Delta^{4}},
$$
values, the number of possible values of $\lambda \in \Lambda$ is bounded by at most twice the same quantity, that is, 
$$
\# \Lambda \le \frac{N^{3+o(1)}}{\Delta^{4}}.
$$
Substituting into~\eqref{eq:XXXv} we obtain 
\begin{equation} \label{lastcase}
\Delta^4 \# \Lambda  \le N^{3+o(1)}.
\end{equation}
Combining~\eqref{eq:CASE177} and~\eqref{lastcase} with~\eqref{eq:XXXv}, we conclude the proof. 
\end{proof}

\begin{lemma}
\label{lem:EqN}
For a weight $\bbeta$ as in~\eqref{eq:weight} supported on $[N,2N]$ with $2N \le q$, we have 
$$
E_{q,j}(\bbeta)\ll  \|\bbeta\|_{\infty}^{8/3}\|\bbeta\|_1^{4/3}\(\frac{N^{13/6}}{q^{1/2}}+N\)N^{o(1)}.
$$
\end{lemma}

\begin{proof}
Recalling~\eqref{eq:E Q2}, we see that 
\begin{equation}
\label{eq:EtoF*}
E_{q,j}(\bbeta)=\sum_{\lambda \in \Fq^{*}}|Q_{\lambda,j}(\bbeta)|^2+O(\|\bbeta\|_2^4).
\end{equation}
By the  H\"{o}lder inequality  we have 
\begin{equation}
\begin{split}
\label{eq:eto4h}
\sum_{\lambda \in \Fq^{*}}|Q_{\lambda,j}(\bbeta)|^2 \le \left(\sum_{\lambda \in \Fq^{*}}|Q_{\lambda,j}(\bbeta)| \right)^{2/3}\left(\sum_{\lambda \in \Fq^{*}}|Q_{\lambda,j}(\bbeta)|^4 \right)^{1/3}.
\end{split} 
\end{equation}
By the triangle inequality we have 
\begin{equation}
\label{eq:betaell1}
\sum_{\lambda \in \Fq^{*}}|Q_{\lambda,j}(\bbeta)|=O(\|\bbeta\|_1^2).
\end{equation}
We also have the trivial inequality 
$$
Q_{\lambda,j}(\bbeta)\le \|\bbeta\|_{\infty}^2.  Q_{\lambda,j}\(\mathbf{1}_{[N,2N]}\). 
$$
Therefore, by Lemma~\ref{lem:EB}, we have
\begin{equation}
\label{eq:betaell4}
\sum_{\lambda \in \Fq^{*}}|Q_{\lambda,j}(\bbeta)|^4\le  \|\bbeta\|_{\infty}^8\left(\frac{N^{13/2+o(1)}}{q^{3/2}}+N^{3+o(1)}\right).
\end{equation}
Substituting~\eqref{eq:betaell1} and~\eqref{eq:betaell4}  in~\eqref{eq:eto4h} we obtain 
\begin{equation}
\label{eq:Qb77}
\sum_{\lambda \in \Fq^{*}}|Q_{\lambda,j}(\bbeta)|^2\le \|\bbeta\|_{\infty}^{8/3}\|\bbeta\|_1^{4/3}\left(\frac{N^{13/6}}{q^{1/2}}+N\right)N^{o(1)}.
\end{equation}
Using~\eqref{eq:EtoF*} we complete the proof after observing 
$$
\|\bbeta\|_2^4\ll \|\bbeta\|_{\infty}^{8/3}\|\bbeta\|_1^{4/3}N^{2/3},
$$
which is always dominated by the upper bound in~\eqref{eq:Qb77}.
\end{proof}

We have another bound which does better for weights supported on longer intervals. For example, 
when $\bbeta$ satisfies~\eqref{eq:balanced} and also $ \| \bbeta \|_1 = N^{1+o(1)}$ our next bound  is stronger for $N \ge q^{1/2}$.

\begin{lemma}
\label{lem:EqN-2}
For a weight $\bbeta$ as in~\eqref{eq:weight} supported on $[N,2N]$ with $2N \le q$, we have 
$$
E_{q,j}(\bbeta)  \ll \| \bbeta \|^2_{\infty} \| \bbeta \|^2_1 (N^2/q + N^{1/2} q^{o(1)} ) . 
$$
\end{lemma}

\begin{proof}  From~\eqref{eq:EtoF*} and~\eqref{eq:betaell1} we derive
\begin{equation} \label{master}
|E_{q,j}(\bbeta)| \leq \max_{\lambda \in \Fq^{\times}} |Q_{\lambda,j}(\bbeta) | \cdot \sum_{\lambda \in \Fq^{\times}}  |Q_{\lambda,j}(\bbeta) |+O \( \| \bbeta \|^{4}_2 \).
\end{equation}

The triangle inequality gives
\begin{equation} \label{sumCR}
 \sum_{\lambda \in \Fq^{\times}}  |Q_{\lambda,j}(\bbeta) |\ll  \| \bbeta \|^2_1 .
\end{equation}
Now,
\begin{equation} \label{intermed}
\max_{\lambda \in \Fq^{\times}} |Q_{\lambda,j}(\bbeta)|  \leq \| \bbeta \|_{\infty}^2 \max_{\lambda \in \Fq^{\times}} Q_{\lambda,j}(\mathbf{1}_{[N,2N]}).
\end{equation}

Next, we  show that  
 \begin{equation} 
 \begin{split}
 \label{count}
 Q_{\lambda,j}(\mathbf{1}_{[N,2N]}) \le  4  \cdot \#  \bigl \{ (Z,V) \in [-2N,2N]  \times & [N,2N]:\\
  (Z-j \lambda^2)^2&=4 \lambda^2 jV \bigr\}. 
 \end{split}
\end{equation}
Indeed, recall that 
$$
Q_{\lambda,j}(\mathbf{1}_{[N,2N]})= \sum_{\substack{ (u,v) \in \Fq^2 \\ u - v = \lambda \\ ju^2, jv^2 \in [N,2N] }} 1.
$$
Making a change of variables 
$$
U=ju^2 \quad \text{and} \quad V=jv^2, 
$$
and using the linear equation $u-v=\lambda$, we see that 
$$
U-V=j \lambda(2v+\lambda).
$$
Rearranging and squaring, we obtain
$$
(U-V-j \lambda^2)^2=4 \lambda^2 jV. 
$$
Making the linear change in variables 
$$
Z:=U-V, 
$$
we obtain 
\begin{equation} \label{quadvar}
(Z-j \lambda^2)^2=4 j \lambda^2 V.
\end{equation}
Given any solution $(Z,V)$ to~\eqref{quadvar}, this corresponds to at most 4 pairs $(u,v) \in \mathbb{F}^2_q$, and so this establishes~\eqref{count}.  
We apply~\cite[Theorem~5]{CCGHSZ}, which in the special case of 
quadratic polynomials implies
 $$
Q_{\lambda,j}(\mathbf{1}_{[N,2N]}) \ll N^2/q + N^{1/2}  q^{o(1)}
$$
 uniformly with respect to $j,\lambda \in \Fq^{\times}$.  Combining this bound with~\eqref{intermed} gives
\begin{equation} \label{CGOSbd}
 \max_{\lambda \in \Fq^{\times}} |Q_{\lambda,j}(\bbeta) |  \ll \| \bbeta \|^2_{\infty} \(N^2/q + N^{1/2} q^{o(1)} \).
\end{equation}
Using~\eqref{sumCR} and~\eqref{CGOSbd} in~\eqref{master} we obtain 
$$
|E_{q,j}(\bbeta)| \ll \| \bbeta \|^2_1 \| \bbeta \|^2_{\infty}  (N^2/q + N^{1/2} q^{o(1)})  + \| \bbeta \|^4_2.
$$
Since trivially
$$
\| \bbeta \|_2 \le  \| \bbeta \|_1^{1/2}   \| \bbeta \|_{\infty}^{1/2}
 $$ 
we now derive the desired result. 
\end{proof}

We further remark that for small $N$ yet another bound on  additive energy of   modular square roots is possible,
improving that of Lemmas~\ref{lem:EqN} and~\ref{lem:EqN-2}.
This bound does not improve our main results but since the question is of independent interest 
we present it in Appendix~\ref{app:A}.

\subsection{Bounding bilinear sums via additive energy} 
Applying the Cauchy--Schwarz
inequality, expanding the second square and  then interchanging summation, gives 
$$
|W_{a,q}(\balpha, \bbeta; h,M,N)|^2  \leq   \| \balpha  \|^2_2  \sum_{n_1,n_2\sim N}  \beta_{n_1} \overline{ \beta}_{n_2}  \sum_{m \sim M}  \sum_{\substack{u,v \in \Fq \\u^2 = amn_1\\v^2 = amn_2}} \eq(h(u-v)).
$$
We now write  
\begin{equation}
\label{eq:WRR}
|W_{a,q}(\balpha, \bbeta; h,M,N)|^2  \le \| \balpha \|^2_2 (R_1 + R_{-1})
\end{equation}
where 
$$
R_j:= \sum_{\substack{ n_1,n_2\sim N \\  \(\frac{n_1}{q} \) = \(\frac{n_2}{q} \) = j  }} \beta_{n_1} \overline{ \beta}_{n_2}  \sum_{\substack{m \sim M \\ \(\frac{am}{q} \) =j }} \hspace{0.1cm}\sum_{\substack{u,v \in \Fq \\u^2 = amn_1\\v^2 = amn_2}} \eq(h(u-v)).
$$

Both sums can be estimated analogously, so we only concentrate on $R_1$. Simplifying $R_1$, we obtain 
$$
R_1=\frac{1}{2} \sum_{n_1,n_2\sim N} \beta_{n_1} \overline{\beta}_{n_2}  \sum_{m \sim M} \,  \sum_{\substack{t \in \Fq \\ t^2 = am}} \, \sum_{\substack{u,v \in \Fq \\u^2 = n_1\\v^2 = n_2}}
\eq(ht (u-v)).
$$
Collecting the terms with the same value of $\lambda = u-v$ gives
$$  
R_1= \frac{1}{2}  \sum_{\lambda \in \Fq} A_{h,\lambda,a}  Q_{\lambda,1}(\bbeta), 
$$
where 
$$
A_{h,\lambda,a} :=    \sum_{m \sim M}  \sum_{\substack{t \in \Fq \\ t^2 = am}}  \eq(ht\lambda)
$$ 
and $ Q_{\lambda,1}(\bbeta)$ is defined in~\eqref{Qdef}.

We  also notice that 
 \begin{equation}
\label{eq:A4}
 \sum_{\lambda \in \Fq} |A_{h,\lambda,a}|^4 \ll q E_{q,b}\(\mathbf{1}_{[M,2M]}\), 
\end{equation}
where $b$ is defined by $ab \equiv 1 \pmod q$, $1 \le b < q$. 

Thus, writing $|Q_{\lambda,1}(\bbeta)| = \(|Q_{\lambda,1}(\bbeta)|^2\)^{1/4} |Q_{\lambda,1}(\bbeta)|^{1/2},$  by the H{\"o}lder inequality, 
using~\eqref{eq:betaell1} (the bound~\eqref{eq:betaell1} also holds when sum is over all $\lambda \in \mathbb{F}_q$) and~\eqref{eq:A4} we derive
 \begin{equation}
 \begin{split}
\label{eq:RQ2}
|R_1|^4  & \ll q E_{q,b}\(\mathbf{1}_{[M,2M]}\) \sum_{\lambda \in \Fq} |Q_{\lambda,1}(\bbeta)|^2 \(\sum_{\lambda \in \Fq} |Q_{\lambda,1}(\bbeta)|\)^2\\
& \ll q  \| \bbeta \|^4_1 E_{q,b}\(\mathbf{1}_{[M,2M]}\)   E_{q,1}\(\bbeta\), 
\end{split}
\end{equation}
as well as a full analogue of~\eqref{eq:RQ2} for $R_{-1}$.

Now using  Lemma~\ref{lem:EqN}, we obtain 
 \begin{equation}
 \begin{split}
\label{eq:R-fin}
|R_1|^4   
\le   q^{1+o(1)}   \|\bbeta\|_{\infty}^{8/3}&\|\bbeta\|_1^{4+4/3}\(\frac{N^{13/6}}{q^{1/2}}+N\)  \\
& \qquad \quad \(M^{7/2}q^{-1/2} + M^{7/3}\) , 
\end{split}
\end{equation} 
which now together with~\eqref{eq:WRR}
implies the first bound of Theorem~\ref{thm:Waq}.  

To obtain the second bound of Theorem~\ref{thm:Waq}, we follow the same argument but use Lemma~\ref{lem:EqN-2}
instead of Lemma~\ref{lem:EqN}, which concludes the proof.

\begin{remark} By orthogonality we have, 
$$
 \sum_{\lambda \in \Fq} |A_{h,\lambda,a}|^2 \ll q M.
$$
Hence, instead of~\eqref{eq:RQ2} we have
$$
|R_1|^2  \ll q M \sum_{\lambda \in \Fq} | Q_{\lambda,1}(\bbeta)|^2.
$$
From Lemma~\ref{lem:EqN} we now derive
\begin{align*}
|W_{a,q}(\balpha, \bbeta&; h,M,N)|\\
& \le 
q^{1/4}M^{1/4} \| \balpha \|_2  \|\bbeta\|_{\infty}^{2/3}\|\bbeta\|_1^{1/3}\(\frac{N^{13/24}}{q^{1/8}}+N^{1/4}\)N^{o(1)}\end{align*}
and similarly we can get yet another bound using Lemma~\ref{lem:EqN-2}.
\end{remark}

We also discuss some alternative approaches to bound bilinear sums  in Appendix~\ref{app:B}.

\section{Equidistribution of roots of primes: Proof of Theorem~\ref{thm:disc p} }

\subsection{Preliminary transformations} 
For $P \leq q$, let $\cP_{q}(P)$ be the set of primes $p\le P$, $p\ne q$, such that $p$ is a quadratic residue modulo $q$. To study the discrepancy of the roots of these quadratic congruences, we introduce the exponential sum
$$
S_{q}(h,P):= \sum_{p \in \cP_{q}(P)} \sum_{\substack{x \in \Fq \\ x^2 = p}} \eq(hx) =  \sum_{p \le P} \sum_{\substack{x \in \Fq \\ x^2 = p}} \eq(hx) +O(1)
$$
(where the term $O(1)$ acoounts for $q=p$ which has possibly been added to the sum). We see that Lemma~\ref{lem:ET} reduces the discrepancy 
question  to 
estimating the sums $S_{q}(h,P)$. 

In fact, following the standard principle, we also introduce the sums
\begin{equation} \label{vansum}
\widetilde{S}_{q}(h,P) = \sum_{k=1}^P \Lambda(k) \sum_{\substack{x \in \Fq \\ x^2 =k}} \eq(hx)\\
\end{equation}
where, as usual we use $\Lambda$ to denote the {\it von Mangoldt function\/}. Clearly, one can bound the sums $S_{q}(h,P)$ via the sums $\widetilde{S}_{q}(h,t)$, $t \le P$, using 
partial summation.

\subsection{The Heath-Brown identity} 

To estimate the sum~\eqref{vansum} we apply the Heath-Brown identity in the form  given 
by~\cite[Lemma~4.1]{FKM}
(see also~\cite[Proposition~13.3]{IwKow})   as well as a smooth partition of 
unity from~\cite[Lemm\'{e}~2]{Fouv} (or~\cite[Lemma~4.3]{FKM}). 
 
We also fix some parameters $L,H$ satisfying 
\begin{equation} \label{parametercond}
P^{1/2} \ge \newY \ge H \ge 1,
\end{equation}
to be optimised later.  Also define 
\begin{equation}
 \label{def:J}
J = \rf{\log P/\log H}.
\end{equation}
We always assume that $H$ exceeds some fixed small power of $q$ so we always have $J \ll 1$.
 
Now, as in~\cite[Lemma~4.3]{FKM}, we decompose  $\widetilde{S}_q(h,P)$ into a linear combination of $O(\log^{2J} q)$ sums
with coefficients bounded by $O(\log q)$,
\begin{align*}
\Sigma(\mathbf{V}):=\sum_{m_1, \ldots, m_J=1}^{\infty} &\gamma_1(m_1)\cdots  \gamma_J(m_J)  \sum_{n_1,\ldots , n_J=1}^\infty \\
\\& V_1 \( \frac{n_1}{N_1} \)  \cdots V_J \( \frac{n_J}{N_J} \)
\sum_{\substack{x \in \Fq \\ x^2=m_1 \cdots m_J n_1 \cdots  n_J}} \e_q(hx), 
\end{align*}  
where  
\begin{equation} \label{eq:cond1}
\mathbf{V}:=(M_1,\ldots , M_J,N_1,\ldots  ,N_J) \in [1/2,2P]^{2J} 
\end{equation}
is a $2J$-tuple of parameters satisfying
 \begin{equation} \label{eq:cond2}
N_1 \geq \ldots \ge N_J, \quad  M_1,\ldots ,M_J \leq P^{1/J},\quad   P \ll  Q \ll  P,
\end{equation}
(implied constants are allowed to depend on $J$),  
 \begin{equation} \label{eq:prod Q}
Q =  \prod_{i=1}^J M_i \prod_{j =1}^JN_j,
\end{equation}  
and
\begin{itemize}
\item the arithmetic functions $m_i \mapsto \gamma_i(m_i)$ are bounded and supported in $[M_i/2,2M_i]$;
\item the smooth functions $x_i \mapsto V_i(x)$ have support in $[1/2,2]$ and satisfy
$$
V^{(j)}(x) \ll  q^{j \varepsilon}
$$
for all integers $j \geq 0$, where the implied constant may depend on $j$ and $\varepsilon$. 
\end{itemize}  

We recall that $a \sim A$ as an equivalent of $a \in [A, 2A]$. Hence we 
can rewrite the sums $\Sigma(\mathbf{V})$ in the following form
\begin{equation}
 \label{eq:SigmaP}
\begin{split}
\Sigma(\mathbf{V})= \sum_{\substack{m_i \sim M_i,  n_i \sim N_i\\ i =1, \ldots J}}  \gamma_1(m_1)\cdots  \gamma_J(m_J) 
 & V_1 \( \frac{n_1}{N_1} \)  \cdots V_J \( \frac{n_J}{N_J} \) \\ 
& 
\sum_{\substack{x \in \Fq \\ x^2=m_1 \cdots m_J n_1 \cdots  n_J}} \e_q(hx). 
\end{split}  
\end{equation} 
In particular, we see that the sums $\Sigma(\mathbf{V})$ are supported on a finite set. 
We now collect various bounds on the sums $\Sigma(\mathbf{V})$ which we derive 
in various ranges of parameters $M_1,\ldots , M_J,N_1,\ldots  ,N_J$ until we cover the whole
range in~\eqref{eq:cond1}.

\subsection{Bounds of multilinear sums} 

To estimate the multilinear sums $\Sigma(\mathbf{V})$, we put $N_1$ in ranges which we call ``{\it{small\/}}'', ``{\it{moderate\/}}", ``{\it{large\/}}" and  ``{\it{huge\/}}''.
We further split the ``{\it{moderate\/}}" range in subranges depending on
``{\it{small\/}}" and ``{\it{large\/}}" values of $N_2$.

Our main tool is the bound~\eqref{eq:simple2}. 
It is convenient to assume that 
\begin{equation}
 \label{eq: P large}
P\ge q^{13/20},
\end{equation}  
as otherwise the result is trivial. We first note a general estimate. Let $\cI,\cJ\subseteq \{1,\dots,J\}$ and write 
\begin{equation}
\label{eq:MNcase1}
M=\prod_{i\in \cI}M_i\prod_{j\in \cJ}N_j, \quad N=Q/M.
\end{equation}
Grouping variables in~\eqref{eq:SigmaP} according to $\cI,\cJ$, there exists $\alpha,\beta$ satisfying 
$$
\|\alpha\|_{\infty}, \|\beta\|_{\infty}=Q^{o(1)},
$$ 
such that 
$$
\Sigma(\mathbf{V})=\sum_{\substack{m\le 2^J M \\ n\le 2^J N}}\alpha(m)\beta(n)\sum_{\substack{x\in \mathbb{F}_q \\ x^2=mn}}e_q(hx).
$$
By~\eqref{eq:simple2} we have 
$$
 \Sigma(\mathbf{V})\ll  q^{1/8}(MN)^{19/24+o(1)}\left(\frac{M^{7/48}}{q^{1/16}}+1\right)\left(\frac{N^{7/48}}{q^{1/16}}+1\right).
$$
Using~\eqref{eq:MNcase1} this simplifies to
\begin{equation}
\label{eq:Sigma77} 
\begin{split} 
 \Sigma(\mathbf{V}) &\ll  Q^{15/16+o(1)}+q^{1/8}Q^{19/24+o(1)} \\ & +\frac{q^{1/16}Q^{15/16+o(1)}}{N^{7/48}}+q^{1/16}Q^{19/24+o(1)}N^{7/48}.
\end{split} 
\end{equation}
 
\subsection{Case I: Small $N_1$} 
First we consider the case when 
\begin{equation}
 \label{eq: small N1}
 N_1  \le H.
\end{equation}

From the definition of $J$ in~\eqref{def:J} and the condition~\eqref{eq:cond2}  we see that 
\begin{equation}
 \label{eq: small M}
M_1,\ldots ,M_J \leq  H.
\end{equation}

We see that if~\eqref{eq: small N1} holds then we 
can  choose two arbitrary sets $\cI, \cJ \subseteq\{1, \ldots, J\}$ such that for 
$$
M = \prod_{i\in \cI} M_i \prod_{j \in \cJ} N_j \mand N = Q/M, 
$$
where $Q$ is given by~\eqref{eq:prod Q} and we have 
 \begin{equation}
 \label{eq: N q12}
P^{1/2} \ll N \ll H^{1/2}P^{1/2}. 
\end{equation}
Indeed, we simply start multiplying consecutive elements of the sequence $M_1,\ldots , M_J,N_1,\ldots  ,N_J$
until the product $Q_+$ exceeds $P^{1/2}$ while the previous product $Q_- < P^{1/2}$. 
Since by~\eqref{eq: small N1} and~\eqref{eq: small M} each factor is at most $H$, we have $Q_+ < H Q_-$.  
Hence  
\begin{itemize}
\item either we have $P^{1/2} \le Q_+ \le P^{1/2}H^{1/2}$ and then we set 
 $M= Q/Q_+$ and $N = Q_+ $;
\item  or we have $P^{1/2} >Q_- > H^{-1/2} P^{1/2}$ and then we set 
$M = Q_-$ and $N= Q/Q_-$. 
\end{itemize}
Hence in either case  the corresponding  $N$ satisfies the 
upper bound in~\eqref{eq: N q12}.  

In this case, since for $N \gg P^{1/2}$ we have $M \ll P/N \ll  P^{1/2}$.
Using~\eqref{eq:Sigma77} this becomes
\begin{equation}
\label{eq:case1final}
\Sigma(\mathbf{V})   \ll Q^{15/16+o(1)}+q^{1/8}Q^{19/24+o(1)}+q^{1/16}Q^{83/96+o(1)}H^{7/96}.
\end{equation}

\subsection{Case II: Moderate $N_1$}
Suppose 
$$
H<N_1\le \newY .
$$

If $N_2<H$ we may argue as in Case I to obtain the bound~\eqref{eq:case1final}. Hence we may suppose 
$$
H\le N_2 \le N_1\le \newY .
$$
In this case we define $M,N$ as 
$$
M=\prod_{i=1}^{J}\prod_{j=3}^{J}N_j \quad \text{and} \quad N=N_1N_2,
$$
so that 
$$
H^2\le N \le \newY ^2.
$$
By~\eqref{eq:Sigma77} we have 
\begin{equation}
\begin{split} 
\label{eq:case2final} 
\Sigma(\mathbf{V}) &\ll  Q^{15/16+o(1)}+q^{1/8}Q^{19/24+o(1)} \\ &\qquad  +\frac{q^{1/16}Q^{15/16+o(1)}}{H^{7/24}}+q^{1/16}Q^{19/24+o(1)}\newY ^{7/24}.
\end{split} 
\end{equation}

\subsection{Case III: Large $N_1$}
Consider next when 
$$
\newY \le N_1\le \newY ^2.
$$
In this case we set
$$
M=\prod_{i=1}^{J}M_i\prod_{j=2}^{J}N_j, \quad N=N_1.
$$
Using~\eqref{eq:Sigma77}
\begin{equation}
\label{eq:case3final} 
\begin{split} 
 \Sigma(\mathbf{V}) &\ll  Q^{15/16+o(1)}+q^{1/8}Q^{19/24+o(1)} \\ 
& \qquad +\frac{q^{1/16}Q^{15/16+o(1)}}{\newY ^{7/48}}+q^{1/16}Q^{19/24+o(1)}\newY ^{7/24}.
\end{split} 
\end{equation}

\subsection{Case IV: Huge $N_1$} 
 We now consider the case  when
\begin{equation}
 \label{eq: huge N1}
\newY ^2 <N_1 \le P.
\end{equation}  
In this case we write 
\begin{align*} 
\Sigma(\mathbf{V})  = \sum_{\substack{m_i \sim M_i \\ i =1, \ldots, J}} \gamma_1(m_1)\cdots  \gamma_J(m_J)  
& \sum_{\substack{ n_i \sim N_i\\ i =2, \ldots , J}}V_2 \( \frac{n_2}{N_2} \)  \cdots V_J \( \frac{n_J}{N_J} \) \\
& \sum_{n_1\sim N_1} V_1\( \frac{n_1}{N_1} \)
\sum_{\substack{x \in \Fq \\ x^2=m_1 \cdots m_J n_1 \cdots  n_J}} \e_q(hx). 
\end{align*}

For each fixed choice of $m_1,\ldots, m_J$ and $n_2, \ldots, n_J$ we set 
$$a=m_1\cdots m_J n_2 \cdots n_J,
$$ 
and we complete the innermost summation over $n_1$
using the standard completion technique, see~\cite[Section~12.2]{IwKow}. 
More precisely, partial summation gives 
\begin{align*}
 \sum_{n_1\sim N_1}  V_1 \( \frac{n_1}{N_1} \)  &\sum_{\substack{x \in \Fq \\ x^2=a n_1}} \e_q(hx)\\
 & =- \frac{1}{N_1} \int_{N_1/2}^{2 N_1} V^{\prime} \( \frac{u}{N_1} \)   \sum_{N_1/2 \leq w<u} \sum_{\substack{x \in \Fq \\ x^2=a w}} \e_q(hx)   du.
\end{align*}
 
Recalling Lemma~\ref{lem:Incom Sqrt}, we conclude that
$$
\sum_{n_1\sim N_1} V_1 \( \frac{n_1}{N_1} \)  \sum_{\substack{x \in \Fq \\ x^2=a n_1}} \e_q(hx) \ll q^{1/2+o(1)}.
$$ 
Therefore, by~\eqref{eq: huge N1}
\begin{equation}
\begin{split} 
 \label{eq: Bound4}
\Sigma(\mathbf{V}) & \ll M_1 \cdots  M_J N_2 \cdots  N_J  q^{1/2+o(1)}  \\
&= P  N_1^{-1}   q^{1/2+o(1)} \le  \frac{P   q^{1/2+o(1)}}{L^2}. 
\end{split} 
\end{equation}

\subsection{Optimisation and concluding the proof}  
Now, combining the bounds~\eqref{eq:case1final}, \eqref{eq:case2final}, \eqref{eq:case3final} and~\eqref{eq: Bound4} we obtain
$$
\Sigma(\mathbf{V}) \ll \(Q^{15/16}+q^{1/8}Q^{19/24} + T_1 + T_2 + U_1+U_2 + U_3\) Q^{o(1)} , 
$$
where 
\begin{equation}
 \label{eq: TTerms}
T_1 = q^{1/16}Q^{83/96}H^{7/96},\qquad T_2 = \frac{q^{1/16}Q^{15/16}}{H^{7/24}}
\end{equation}
and 
\begin{equation}
 \label{eq: UTerms}
U_1 = \frac{q^{1/16}Q^{15/16}}{\newY ^{7/48}}, \quad U_2 = q^{1/16}Q^{19/24}\newY ^{7/24}, \quad U_3 =\frac{q^{1/2}Q}{\newY ^2}.
\end{equation}

First we choose $H$ to balance $T_1$ and $T_2$ in~\eqref{eq: TTerms}   and choose $\newY $   to balance $U_2$ and $U_3$ in~\eqref{eq: UTerms}.
 This gives:
\begin{equation} \label{HLchoice}
H=Q^{1/5} \mand \newY =q^{21/110}Q^{1/11}.
\end{equation}
Recalling~\eqref{eq: P large}, we have $H\ge q^{13/100}$ and 
hence $J \ll 1$ as required. Using the fact that $Q\ll P$, we see that the conditions in \eqref{parametercond} 
are satisfied provided 
$$
q^{7/15} \le Q\le q^{7/4}.
$$
Recalling~\eqref{eq: P large}, without loss of generality we may assume
$$
q^{13/20} \le Q \le q^{7/4}.
$$

The choice of $H$ and $\newY $ in~\eqref{HLchoice} and the fact that $Q\ll P$, gives 
\begin{align*}
\Sigma(\mathbf{V})&\ll P^{15/16+o(1)}+q^{1/8}P^{19/24+o(1)}+q^{1/16}P^{211/240+o(1)} \\
&\qquad \qquad \qquad +q^{61/1760}P^{61/66+o(1)}+q^{13/110}P^{9/11+o(1)}.
\end{align*} 
Noting that for $q^{13/20} \le P\le q$, either the last or second last term dominates the sum, this simplifies to 
$$
\Sigma(\mathbf{V}) \ll q^{61/1760}P^{61/66+o(1)}+q^{13/110}P^{9/11+o(1)},
$$
which completes the proof of Theorem~\ref{thm:disc p}.

\begin{remark}\label{rem:U1U2} 
One can also choose $L = q^{1/3}$ to balance $U_1$ and $U_2$ in~\eqref{eq: UTerms}. This sometimes 
gives a better bound but only for $P \ge q^{3/4}$, while, as we have mentioned, we are interested in as small as possible 
values of $P$.
\end{remark}

\appendix 

\section{Additive energy bounds of modular square roots} 
\label{app:A}

For small $N$ we  have an  improvement of Lemma~\ref{lem:EqN}.
To emphasise the ideas we consider the following special case of the quantity $E_{q,j}(\bbeta)$ from Section~\ref{sec:E-bounds}.
 For an integer $N$ we define  
$$
E_q(N) := \#\{ (u,v,x,y) \in \Fq:~ u^2,v^2, x^2, y^2 \sim N \ \text{and} \ u + v = x+y\}
$$
(recall that $u^2,v^2, x^2, y^2$ are  all computed modulo $q$). 

\begin{prop}
\label{prop:EqN-small}
For any  positive integer $ N \le q^{1/2}$, we have 
$$
E_q(N)  \le  N^6q^{-1 + o(1)}  + N^2q^{o(1)}.
$$
\end{prop}

\begin{proof}  
Squaring both sides of the congruence $ u + v \equiv  x+y \pmod q$,  we obtain 
 \begin{equation} 
 \label{eq:congr2}
u^2 + 2 uv  +v^2  \equiv  x^2 +2xy+y^2 \pmod q.  
\end{equation}
We denote 
$$ 
 z := x^2 +y^2 - u^2 -v^2, 
$$
and write the  congruence~\eqref{eq:congr2} as 
$$
 2 (uv -xy)   \equiv  z \pmod q
$$
with $z \in [-6N, 6N]$.
We square it again, and obtain
$$
8 uvxy \equiv w \pmod q
$$
with 
$$
w =  4u^2 v^2 + 4x^2 y^2 -z^2 \in [-8 N^2, 8N^2].
$$
Since $u^2,v^2, x^2, y^2 \sim N$ and $z \in [-6N, 6N]$, we see that 
$$
w \in  [-4N^2, 32N^2].
$$
Thus the product $uvxy$ falls in $O(N^2)$ arithmetic progressions modulo $q$ 
and thus so does the product $UVXY \le 256 N^4$, where $U$, $V$, $X$ and $Y$ are the smallest 
positive residues modulo $q$ of 
$u^2$, $v^2$, $x^2$ and $y^2$, respectively. Thus the product   $UVXY$ can take at most $O\(N^2(N^4/q+1)\)$
possibilities, each of them leads to $q^{o(1)}$ possibilities for individual values $(U,V,X,Y)$ and thus for 
the initial variables $(u,v,x,y)$. 
\end{proof}

Note that for $N \le q^{1/4+o(1)}$,  Proposition~\ref{prop:EqN-small} implies an essentially optimal 
bound $E_q(N)  \le  N^{2+o(1)}$. 

\section{Some related sums}
\label{app:B}

\subsection{Type-I and Type-II sums} 
The sums $V_{a,q}(\balpha, \bbeta; h,M,N)$ with two nontrivial weights are usually   called {\it Type-II\/} sums.

However for some applications  sums with only one nontrivial weight, such as 
$$
V_{a,q}(\balpha; h,M,N) =   \sum_{m \sim M}  \sum_{n \sim N} \alpha_m   \sum_{\substack{x \in \Fq \\
x^2 = amn}} \eq(hx), 
$$
are  also important and are  usually   called {\it Type-I\/} sums.

Typically Type-I sums admit an easier treatment than Type-II sums, and with stronger bounds.  For  the sums
$V_{a,q}(\balpha, \bbeta; h,M,N)$, one can apply  some ideas 
of Blomer, Fouvry,  Kowalski,  Michel,  and Mili\'{c}evi\'{c}~\cite{BFKMM} with a follow up application of the 
Bombieri bound~\cite{Bom} for exponential sums along a curve. Unfortunately the resulting bound
\begin{equation}
 \label{eq: Bound V}
|V_{a,q}(\balpha; h,M,N)|  
\le  \sqrt{ \| \boldsymbol{\alpha} \|_1  \| \boldsymbol{\alpha} \|_2} M^{1/12} N^{7/12} q^{1/4+o(1)}, 
\end{equation} 
which also requires the additional conditions
\begin{equation} \label{type1cond}
MN \leq q^{3/2} \mand M \leq N^2,
\end{equation}
 does not   improve a combination of  Theorem~\ref{thm:Waq} and the bound
 $$
 |V_{a,q}(\balpha; h,M,N) | \le Mq^{1/2+o(1)},
 $$
 which can be obtained via the completion method exactly as in~\eqref{eq: Bound4}.
 Since the argument may have other applications we 
sketch it here with a brief outline of the main steps. We then give an short outline of further 
modifications which can achieved within this approach and to which the method of proof of  
 Theorem~\ref{thm:Waq} does not apply.   We also believe that Proposition~\ref{prop:sigma} below 
 is of independent interest and may have further applications. 

\subsection{Preliminary transformations} Let $K: \Fq \rightarrow \mathbb{C}$ be an arbitrary function with $|K(x)| \ll 1$, 
which is usually called the \textit{kernel}.
Consider the sum 
$$
\cS =\sum_{m\sim  M} \alpha_m \sum_{n \sim N} K(mn).
$$
Choose real parameters $A,B \geq 1$ such that 
\begin{equation} \label{restr}
AB \leq N, \quad \text{and} \quad  2AM<q.
\end{equation} 
As in~\cite[Equation~(5.8)]{BFKMM} we have 
\begin{equation}
\begin{split}
\label{ABsum}
AB \cS & \le 
\sqrt{ \| \boldsymbol{\alpha} \|_1  \| \boldsymbol{\alpha} \|_2 } (AN)^{3/4} q^{o(1)} \\
&\qquad \quad \times \( \sum_{r \in \Fq}  \sum_{1 \le s \leq 2AM} \left| \sum_{B<b \leq 2B} \eta(b) K \( s(r+b) \)  \right|^{4} \)^{1/4}, 
\end{split}
\end{equation}
where  $\eta(b)$ are complex numbers satisfying $| \eta(b)| \leq 1$, $b \sim B$. 
Expanding the fourth power in~\eqref{ABsum}, the innermost sum in the second factor becomes 
\begin{equation}
\label{eq:S4}
 \sum_{r \in \Fq}  \sum_{1 \le s \leq 2AM} \left| \sum_{B<b \leq 2B} \eta(b) K \( s(r+b) \)  \right|^{4} = \sum_{\boldsymbol{b} \in \mathcal{B}} \eta(\boldsymbol{b}) \Sigma(K,\boldsymbol{b}),
\end{equation}
where $\mathcal{B}$ denotes the set of quadruples $\boldsymbol{b}=(b_1,b_2,b_3,b_4)$ of integers satisfying $B<b_j \leq 2B$, $j=1,2,3,4$,  the coefficients $\eta(\boldsymbol{b})$
satisfy $| \eta(\boldsymbol{b})| \leq 1$ for all $\boldsymbol{b} \in \mathcal{B}$, and 
\begin{equation}
\begin{split} 
\label{sigmaKb}
\Sigma(K,\boldsymbol{b}):=\sum_{r \in \Fq} \sum_{1 \leq s \leq 2AM}  & K \(s(r+b_1) \)   K \(s(r+b_2) \) \\
& \quad \times \overline{K \(s(r+b_3) \) K \(s(r+b_4) \)}.
\end{split}
\end{equation}

Let $\mathcal{B}^{\Delta}$ be the subset of $\boldsymbol{b}$ admitting a subset of two entries matching the entries of the complement (for instance, such as $b_1=b_2$ and $b_3=b_4$ or 
$b_1=b_3$ and $b_2=b_4$). For such tuples $\boldsymbol{b}$ we use the trivial bound 
\begin{equation} \label{trivbd}
\sum_{\boldsymbol{b} \in \mathcal{B}^{\Delta}} |\Sigma(K,\boldsymbol{b})| \ll AB^2M q.
\end{equation} 

For  $\boldsymbol{b} \not \in \mathcal{B}^{\Delta}$, we complete the sum over $s$ in~\eqref{sigmaKb} using additive characters, see~\cite[Section~12.2]{IwKow}   and derive, similarly to~\eqref{eq: Compl}, 
\begin{equation}
\label{eq:Sigmas}
\Sigma(K,\boldsymbol{b}) \ll \log q \max_{t \in \Fq}  | \Sigma(K,\boldsymbol{b},t)|,
\end{equation}
where
\begin{equation}
\label{Kbh}
\begin{split}
\Sigma(K,\boldsymbol{b},t):=\sum_{r,s \in \Fq} &  \e_q (st) K \(s(r+b_1) \) K \(s(r+b_2) \) \\
& \qquad \quad \times \overline{K \(s(r+b_3) \) K \(s(r+b_4) \)}.
\end{split}
\end{equation}

\subsection{Reduction to exponential sums along a curve}
Now consider the kernel of our interest: 
\begin{equation} \label{kernel}
K(x):=\sum_{\substack{ u \in \Fq \\ u^2=ax}} \e_q(hu).
\end{equation}

\begin{prop} \label{prop:sigma} 
For all $t \in \Fq$ and all $\boldsymbol{b} \not \in \mathcal{B}^{\Delta}$ above we have
$$
\Sigma(K,\boldsymbol{b},t) \ll q.
$$
\end{prop}

\begin{proof} Since  $\boldsymbol{b} \not \in \mathcal{B}^{\Delta}$ there is a least one value among $b_1, b_2, b_3, b_4$ 
which is not repeated among other values. Without loss of generality we can assume 
that 
\begin{equation} \label{eq:b1 unique}
b_1 \ne b_2,b_3,b_4.
\end{equation}

Substituting~\eqref{kernel} into~\eqref{Kbh} we obtain
\begin{equation} \label{sigmasub}
\Sigma(K,\boldsymbol{b},t)= \sum_{r,s \in \Fq} \, \sum_{(u_1,v_1,u_2,v_2) \in \cZ_{\boldsymbol{b},r,s} }  \e_q(u_1+v_1-u_2-v_2+st), 
\end{equation}
where $\cZ_{\boldsymbol{b},r,s} $ is the set of solutions $(u_1,v_1,u_2,v_2)\in \Fq^4$ to 
\begin{equation}
\begin{split} \label{Weqn}
 u_1^2=s(r+b_1), & \quad v_1^2 =s(r+b_2), \\
 u_2^2=s(r+b_3), & \quad v_2^2=s(r+b_4). 
 \end{split}
 \end{equation}
When $r=-b_1$ and $s \in \Fq$, then $\# W_{\boldsymbol{b},r,s}  =O(1)$, and so the contribution to~\eqref{sigmasub} is $O(q)$. 

So now we can assume that $r \neq -b_1$. We see that if $u_1=0$, 
then $s = 0$ and 
so also  $v_1=u_2 = v_2 = 0$. Eliminating $s$ from~\eqref{Weqn}, 
and writing 
$$
(u_1,v_1,u_2,v_2) = (w, wx, wy, wz), 
$$
we see that~\eqref{sigmasub} becomes 
\begin{equation}
\label{eq:Sigma}
\Sigma(K,\boldsymbol{b},t)= \Sigma^*(K,\boldsymbol{b},t)
+ O(q), 
\end{equation}
 where 
$$
 \Sigma^*(K,\boldsymbol{b},t):=\sum_{(r,x,y,z) \in \cV^*_{\boldsymbol{b}} } \sum_{w \in \Fq} 
\e_q\(w(1+x-y-z)+ \overline{(r+b_1)} t w^2\), 
$$
 and the sum is taken over the set  $\cV^*_{\boldsymbol{b}} $ of solutions $(r,x,y,z)\in \Fq^4$  with $r \neq -b_1$ to 
$$
(r+b_1)x^2=(r+b_2),  \quad  (r+b_1) y^2=(r+b_3), \quad (r+b_1)z^2=(r+b_4).
$$ 
We now  change the variable  $r  \to \overline{u}  - b_1$, $u \in\Fq^\times $,  and write 
\begin{equation}
\label{eq:Sigma*}
 \Sigma^*(K,\boldsymbol{b},t)=\sum_{(u,x,y,z) \in \cU^*_{\boldsymbol{b}} } \sum_{w \in \Fq} 
\e_q\(w(1+x-y-z)+   t u w^2\), 
\end{equation} 
where  $\cU^*_{\boldsymbol{b}} $ consists of solutions $(u,x,y,z)\in \Fq^4$  with $u \neq 0$ to 
\begin{equation}
\label{eq:U*}
x^2= 1 + c_1 u ,  \quad    y^2= 1 + c_2 u, \quad  z^2= 1 + c_3 u,
\end{equation} 
where 
$$
c_i = b_{i+1} - b_1, \qquad  i=1,2,3.
$$
Observe that by~\eqref{eq:b1 unique} we have
$$
c_i  \ne 0, \qquad  i=1,2,3.
$$

\subsubsection*{The case $t \in \mathbb{F}^{\times}_q$}
Recalling the definition of the Gauss sums~\eqref{gausseval}
we write 
\begin{equation} \label{eq:Sigma G}
 \Sigma^*(K,\boldsymbol{b},t)=\sum_{(u,x,y,z) \in \cU^*_{\boldsymbol{b}}} \,  \cG_q(t u,1+x -y -z).\end{equation}
Evaluating the Gauss sums as in~\eqref{eval} we see that~\eqref{eq:Sigma G} becomes 
$$
 \Sigma^*(K,\boldsymbol{b},t)=\varepsilon_{q} \sqrt{q} \,  
 \sum_{(u,x,y,z) \in \cU^*_{\boldsymbol{b}} }  \( \frac{t u}{q} \)   \e_q \( - \overline{4 tu}  (1+ x- y - z)^2 \), 
$$
where we can now extend $\cU^*_{\boldsymbol{b}}$ to  the set $\cU_{\boldsymbol{b}}$ which also allows 
the value $u=0$ (this value is eliminated automatically by the pole in the $\overline{4 tu}$ term). 

Let the set $\cU_{\boldsymbol{b}}$ be the same as $\cU^*_{\boldsymbol{b}}$ in  which we also allow 
the value $u=0$ (this value is eliminated automatically by the pole in the $\overline{4 tu}$ term),    
and let  $\cW_{\boldsymbol{b},t}$ be the set of solutions $(w,x,y,z) \in \mathbb{F}_q^4$ to 
$$
x^2=1+ c_1 \overline{4t} w^2,    \quad y^2=1+ c_2 \overline{4t} w^2, \quad 
z^2=1+ c_3 \overline{4t} w^2.
$$
Recalling the definition of  the Legendre symbol,  we represent the mixed sum $\Sigma^*(K,\boldsymbol{b},t)$ with multiplicative and additive characters as a linear 
combination  of two pure exponential sums with rational functions
\begin{equation}
\label{eq:Sigma12}
 \Sigma^*(K,\boldsymbol{b},t):=2 \Sigma_2 (K,\boldsymbol{b},t)-\Sigma_1(K,\boldsymbol{b},t), 
\end{equation}
where
\begin{align*}
\Sigma_1(K,\boldsymbol{b},t)&:=\varepsilon_{q} \sqrt{q}   \sum_{(u,x,y,z) \in \cU_{\boldsymbol{b}} }\e_q \( -\overline{4tu}(1+ x - y - z)^2 \),\\
\Sigma_2(K,\boldsymbol{b},t)&:=\varepsilon_{q} \sqrt{q} \,   
\sum_{(w,x,y,z) \in \cW_{\boldsymbol{b},t} }\e_q \( -\overline{w}^2(1+ x - y -z)^2 \).
\end{align*}

A simple argument using the Weil bound on character sums (see, for example,~\cite[Theorem~11.23]{IwKow}) shows that each of the varieties $\cW_{\boldsymbol{b},t}$ and $\cU_{\boldsymbol{b}}$
has $A(c_1,c_2,c_3)q+ O(q^{1/2})$ rational points over $\Fq$
where $A(c_1,c_2,c_3)=1$ if all $c_i$ distinct, $2$ if exactly two of the $c_i$ are equal and 
 $4$ if all the $c_i$ are equal.
 Thus the variety is 
of dimension 1 by the Lang--Weil theorem (that is, an algebraic curve over $\mathbb{F}_q$).
Elementary but somewhat tedious calculations show that the function 
$\overline{4tu}(1+ x - y - z)^2$ is not constant on  $\cU_{\boldsymbol{b}}$ and  the  function $\overline{w}^2(1+ x -y -z)^2$ is not 
 constant on  $\cW_{\boldsymbol{b},t}$.

Therefore the bound of Bombieri~\cite[Theorem~6]{Bom}) applies to  both sums and yields 
$$
\Sigma_{1,2}(K,\boldsymbol{b},t) \ll q.
$$ 
Now using  this bound in~\eqref{eq:Sigma12} and recalling~\eqref{eq:Sigma}, we see that
\begin{equation}
\label{eq:Sigma t}
\Sigma(K,\boldsymbol{b},t) \ll q, \qquad t \in\Fq^\times .
\end{equation}

\subsubsection*{The case $t=0$}
In this case, we see from~\eqref{eq:Sigma*} that
$$
 \Sigma^*(K,\boldsymbol{b},0)=\sum_{(u,x,y,z) \in \cU^*_{\boldsymbol{b}} } \sum_{w \in \Fq} 
\e_q\(w(1+x-y-z)\) = q T,
$$
where $T$ is the number of solutions $(u,x,y,z) \in \mathbb{F}^{4}_q$ to~\eqref{eq:U*} with $1+x= y+z$.  Direct elimination of variables 
shows that $T = O(1)$
and we obtain 
\begin{equation}
\label{eq:Sigma 0}
\Sigma(K,\boldsymbol{b},0) \ll q.
\end{equation}

Combining~\eqref{eq:Sigma t} and~\eqref{eq:Sigma 0}  we  derive the desired result. 
\end{proof}

\subsection{Concluding the argument}
Substituting the bound of   Proposition~\ref{prop:sigma}  in~\eqref{eq:Sigmas} we obtain 
$$
\Sigma(K,\boldsymbol{b}) \ll q \log q, 
$$
with the contribution from $\boldsymbol{b} \in  \mathcal{B}^{\Delta}$ to  in~\eqref{eq:S4}, bounded as~\eqref{trivbd},
we see that~\eqref{ABsum}  becomes 
\begin{align*}
V_{a,q}(\balpha&;  h,M,N)\\
&  \le  (AB)^{-1}  \sqrt{ \| \boldsymbol{\alpha} \|_1  \| \boldsymbol{\alpha} \|_2 } (AN)^{3/4} \left( AB^2 Mq+B^4 q  \right)^{1/4}    q^{o(1)}.
\end{align*}
We now choose
$$
A=\frac{1}{2} M^{-1/3} N^{2/3} \quad \text{and} \quad B=(MN)^{1/3}
$$
to balance the above estimate, and note that the conditions~\eqref{type1cond} imply $A,B \geq 1$ as well as~\eqref{restr},
which implies~\eqref{eq: Bound V}. 

\subsection{Further possibilities}
One of the obvious ways to try to improve the bound~\eqref{eq: Bound V} is to use higher powers as in~\cite{KMS2}. 
However studying exponential sums over more general  higher dimensional varieties can be quite challenging. 

This however leads to some further  possibilities where the above method can be more competitive. One of them  is an extension where the summation over $m$ 
in the sum $V_{a,q}(\balpha, \bbeta; h,M,N)$ from a dyadic interval $m \sim M$ to an arbitrary set $m \in \cM$
with $\cM \subseteq \Fq$. 
More precisely, the method of~\cite{KMS2} rests on bounds for the second moment of the quantity  
$$
\nu(r,s) =  \sumthree_{\substack{ a \sim A, \ m  \sim M, \ n  \in [0, 6N]\\  
am=s, ~\overline{a}n  \equiv r \pmod q}}  |\alpha_m|   .
$$ 
It has been shown in~\cite{BaSh} that one can obtain good bounds on this quantity even if $m$ runs through 
an arbitrary set  $\cM \subseteq \Fq$. 

\section{Correlation between Sali\'{e} sums}
\label{app:C}
The identity~\eqref{eq:S eval} links Sali\'{e} sums to sums over modular square roots and 
plays an important role in the proof of~\cite[Theorem~1.2]{DuZa}. 
Using our argument, we now are able to obtain the following improvement of the bound of~\cite[Theorem~1.2]{DuZa}
 on  sums of  Sali\'{e} sums~\eqref{eq:Sal}.  
 
 \begin{prop} \label{prop:Sal} For any positive integers $M,N\le q$  and any integer $a$ with $\gcd(a,q)=1$, 
we have 
\begin{align*}
 \sum_{n_1,n_2 \sim N}&  \left|  \sum_{m \sim M}   S(m, an_1;q)  S(m,an_2;q)\right| \\
 &   \le \begin{cases} q^{5/4+o(1)}N\(M^{7/8}q^{-1/8} + M^{7/12}\)  \(N^{7/8}q^{-1/8} + N^{7/12} \);\\
 q^{5/4+o(1)} N \(Mq^{-1/4} + M^{5/8}\)  \(Nq^{-1/4} + N^{5/8} \).
 \end{cases} 
\end{align*}
\end{prop} 

\begin{proof} From~\cite[Equations~(8.1) and~(8.2)]{DuZa} we infer that 
$$
\sum_{n_1,n_2 \sim N} \left|  \sum_{m  \sim M}   S(m, an_1;q)  S(m, an_2;q)\right|  \le 
q(\widetilde R_1+\widetilde  R_{-1}).
$$
Here, the $\widetilde R_j$ (with $j = \pm 1$) are as in~\eqref{eq:WRR}, with the only difference being the presence of two-dimensional weights
 $\widetilde  \beta_{n_1,n_2}$ instead of products of two one dimensional weights  $\beta_{n_1} \overline{ \beta}_{n_2}$. This is inconsequential for 
the  argument. In particular,  we have a full analogue of the bound~\eqref{eq:R-fin} and its version using  Lemma~\ref{lem:EqN-2}
 from which we derive the desired result. 
\end{proof}

For example, for $M,N \le q^{2/3}$, the second bound of Proposition~\ref{prop:Sal}  simplifies as
$$
 \sum_{n_1,n_2 \sim N}  \left|  \sum_{m \sim M}   S(m, an_1;q)  S(m,an_2;q)\right| \le MN^2q^{1+o(1)} \(\frac{q^2}{M^3N^3}\)^{1/8}, 
$$
which improves the trivial bound $MN^2q$ whenever $MN \ge q^{2/3+\varepsilon}$ for any fixed $\varepsilon > 0$.
While for $M,N \le q^{3/7}$, the first bound yields
$$
 \sum_{n_1,n_2 \sim N}  \left|  \sum_{m \sim M}   S(m, an_1;q)  S(m,an_2;q)\right| \le MN^2q^{1+o(1)} \(\frac{q^3}{M^5N^5}\)^{1/12}, 
$$
which improves the trivial bound $MN^2q$ whenever $MN \ge q^{3/5+\varepsilon}$ for any fixed $\varepsilon > 0$.

\section*{Acknowledgement} 
The authors thank the anonymous referee for their meticulous comments on the manuscript.
The authors are also very grateful to Bruce Berndt, Moubariz Garaev, Paul Pollack, George Shakan and Asif Zaman 
for their comments on a preliminary version of the manuscript. The authors also thank Paul Pollack for the information 
about the work of Benli~\cite{Ben}.

The work of A.D. was supported on a UIUC
Campus research board grant. 
 The work of  B.K. was  supported  by the Academy of Finland Grant~319180. 
 The work of I.S. was supported in part  by the Australian Research Council Grant~DP170100786.

\end{document}